\documentclass[a4paper,11pt]{amsart}
\usepackage{amsmath,amsfonts,amssymb}
\usepackage{amsthm}
\usepackage{ulem}
\usepackage{graphics,graphicx,epsf}
\usepackage[usenames]{color}
\usepackage{color}
\usepackage{hyperref}
\usepackage{float}
\usepackage{enumerate}
\usepackage[shortlabels]{enumitem}
\usepackage{tikz}
\usetikzlibrary{arrows}
\usepackage{csquotes}

\usepackage[left=2.22cm,right=2.22cm,top=2.9cm,bottom=2.9cm]{geometry} 

\newcommand{\cH}{\mathcal{H}}
\newtheorem{Assumption}{Assumption}[section]
\newtheorem{theorem}{Theorem}[section]
\newtheorem{lemma}[theorem]{Lemma}
\newtheorem{proposition}[theorem]{Proposition}

\newtheorem{remark}[theorem]{Remark}
\newtheorem{definition}{Definition}[section]

\numberwithin{equation}{section}

\numberwithin{equation}{section}

\title[Wave equation with quintic nonlinearities]{Stabilizing energy-critical wave equation to a finite dimensional attractor via nonlinear damping}

\author[Irena Lasiecka]{Irena Lasiecka$^{\star}$}\thanks{$^{\star}$ Research partially supported by National Science Foundation: NSF-DMS-
	2205508.}
\address[Irena Lasiecka]{Department of Mathematical Sciences, The University of Memphis, 3725 Norriswood Ave, Memphis, TN 38152 United States }
\email{lasiecka@memphis.edu}

\author[Vando Narciso]{Vando Narciso$^{\dagger}$}\thanks{$^{\dagger}$ Research partially supported by Fundect/CNPq Grant 15/2024.}
\address[V. Narciso]{Universidade Estadual de Mato Grosso do Sul, 79804-970, Dourados, MS, Brazil}
\email{vnarciso@uems.br}

\date{\today}

\begin{document}

\begin{abstract}
	The wave equation with  energy critical sources and nonlinear damping defined on a 3D bounded domain is considered. It is shown that the resulting dynamical system admits a global attractor. Under the additional assumption of strong monotonicity of the damping at the origin, it is shown that the originally unstable quintic wave is uniformly stabilised to a  finite dimensional and smooth set. Moreover, the existence of an exponential attractor is established. In order to handle \enquote{energy criticality} of both sources and damping, the methods used depend on enhanced dissipation \cite{Bociu-lasiecka-jde}, energy {\it identity} for weak solutions \cite{Koch-lasiecka}, an adaptation of Ball's method \cite{ball}, and the theory of quasi-stable systems \cite{chueshov-white}.

	
\end{abstract}	
\maketitle	
	\section{Introduction}
Let $\Omega \subset \mathbb{R}^3$ be  a bounded domain with  a sufficiently smooth boundary $\Gamma=\partial \Omega$.  The following classical   {\it energy critical}  wave model is under consideration in this work:
\begin{eqnarray}\left\{\begin{aligned}\label{P}	
		&u_{tt}-\Delta u+g(u_t)+f(u)= h,\quad \mbox{in}\quad \Omega\times \mathbb{R}^+,\\
		&u|_{\Gamma}=0,\quad \mbox{on}\quad \Gamma=\partial\Omega,\\
		&u(x,0)=u_0(x),\quad u_t(x,0)=u_1(x),\quad x\in \Omega,\end{aligned}\right.\end{eqnarray}
where $g$ is a nonlinear monotone damping of critical quintic nature: 
$|g(s)|\leq M|s|^5$ and the source $f$ is also a nonlinear critical source with quintic upper bound:
$|f(s)| \leq M|s|^5,\; |s|>1$.
Quintic wave equations, which include nonlinear terms to the fifth power, are used to model complex physical phenomena like solitons in optics and plasma physics, vibrations in materials, fluid dynamics, and the dynamics of wave propagation in various media. They are crucial for understanding and predicting the behavior of nonlinear systems and help in designing technologies in fiber optics, plasma confinement, solid state physics and material science [molecular vibrations and wave propagation in harmonic crystals], and nonlinear circuits. Of particular importance  are cubic-quinic media which describe physical environments where wave wave propagation is governed by two competing types of nonlinear responses. 
The low intensity cubic reaction forces the wave to self-focus, while high intensity quintic reaction activates at intense thresholds  and acts as a physical ceiling that prevents wave collapse, stabilizing localized packets into robust structures refereed to as droplets. 

A study of wave equations in 3 dimensions with a quintic source has attracted a lot of activity. The quintic term is energy-critical for the wave equation in 3 dimensions, thus giving rise to a host of mathematical issues starting with the well-posedness of weak solutions. While local well-posedness can be obtained via several methods, {\it global solvability} has been a source of major challenges. And here is the reason why. The strategy used for establishing existence of weak solutions to the undamped wave equation relies heavily on Strichartz estimates, which allow one to trade differentiability for suitable integrability. However, Strichartz estimates are constructed for each solution [so called Strichartz solution] $u$ on the time interval $[0,T]$. These estimates lose uniformity with respect to bounded sets of initial data when time $T\rightarrow\infty$, unless the nonlinearity is {\it subquintic}. The above predicament, particularly sensitive to the long-term behavior of the models, stimulated a great amount of research -- see \cite{Kalantarov-Savostianov-Zelik}. To address this issue, methods relying on ``nonconcentration'' of singularity in S-estimates have been developed \cite{Burq-Lebeau}. These allow one to extend local solutions to global ones. Additional arguments relying on the so called backward regularity \cite{zelik} lead to S-solutions forming a dynamical system which is eventually shown to have a global attractor. These methods rely on finite speed of propagation, which can also be proved in the presence of {\it linearly bounded dissipation} -- \cite{Kalantarov-Savostianov-Zelik}. However, in the case of highly nonlinear dissipation, proving globality and uniqueness of solutions by using the above-described method appears challenging.

This brings us to the topic of the present paper. Our aim is, in addition to the study of attracting sets, to construct the stabilizing effect of nonlinear damping, which may eventually bring [in the uniform topology of the phase space] the originally unstable hyperbolic dynamics to a {\it finite-dimensional} configuration, characterized by a compact set-attractor. Such a property is of critical value within the realm of control theory, where finite-dimensional control methods could be applied to achieve a desirable outcome for the system. To accomplish this, we consider damping also of critical exponent with an enhanced dissipation effect. The idea goes back to \cite{vitillaro} and \cite{bociu1,Bociu-lasiecka-jde,Lorena-lasiecka}, where such dynamics has been considered. In fact, the enhanced dissipation was used in order to establish well-posedness, particularly uniqueness of weak solutions corresponding to critical and supercritical sources. The estimates in these references provide an initial road map to our analysis, which eventually culminates with the theory of coherent attracting sets. Thus, our goal is to take the problem to the next level by considering long-time behavior--in particular, existence and the structure of attracting sets including {\it finite dimensionality and the existence of exponentially attracting sets}. This will be accomplished by developing appropriate tools for asymptotic smoothness [particularly in the degenerate case when $g'(0)=0$], finite-dimensionality, and maximal regularity of the attractors. The challenge here is due to the double criticality of the source and the dissipation. Criticality of the source makes the potential energy functional $\int_{\Omega} F(u(x))\,dx$, with $F'=f$, noncompact, thus preventing methods of compensated compactness from applicability. Criticality of the dissipation implies $g(u_t)\in H^{-1}(\Omega)$, but $g(u_t)\notin H^{-\alpha}(\Omega)$, $\alpha<1$. The latter is the critical regularity for variational solutions to be well defined with compact test functions in $H^1(\Omega)$. Our final result provides the existence of a global attractor which, moreover, is finite-dimensional and smooth, along with the existence of an exponential attractor. Thus, quintic waves are stabilised uniformly/exponentially to finite-dimensional compact sets.

\section{Functional Setting, Assumptions and Main Results}
\subsection{Functional framework and structural assumptions}
First of all, we will define some notations that will be used throughout the work.  We denote by  $H^s(\Omega)$ the $L^2-$ based Sobolev's space of the order $s$ and by $H^s_0(\Omega)$ the closure of $C^{\infty}_0(\Omega)$ in $H^s(\Omega)$. 
We also denote  
$\left(u,v\right)=(u,v)_{L^2(\Omega)}$ and 
$\|u\|^2=\|u\|^2_{L^2(\Omega)}$ for the inner product and norm in 
$H^0:=L^2(\Omega)$, and 
$((u,v))=\left(\nabla u,\nabla v\right)$ and 
$\|u\|^2_{H^1_0(\Omega)}=\|\nabla u\|^2$
for those in 
$H^1:=H^1_0(\Omega)$.
We consider the Laplacian operator $-\Delta$ defined by the triplet $\{H^1,H^0,a(u,v)\}$ where $a(u,v)$ is the bilinear continuous form on $H^1\times H^1$ given by
$a(u,v)=\left(\nabla u,\nabla v\right)$, $u,v\in H^1.$
Our analysis is carried out in the phase space $
\mathcal H:=\mathcal H_0=H^1\times H^0,
$
associated with weak solutions of problem \eqref{P}, and endowed with the norm
\[
\|U\|_{\mathcal H}^2:=\|\nabla u\|^2+\|v\|^2,
\qquad U=(u,v).
\]
For the construction of regular solutions, we also consider the more regular space
$
\mathcal H_1:=(H^2\cap H^1)\times H^1,
$
equipped with the norm
\[
\|U\|_{\mathcal H_1}^2:=\|\Delta u\|^2+\|\nabla v\|^2.
\]

We assume the following hypotheses imposed on the functions $h$, $g$ and $f$.
\begin{Assumption}\label{Assumption} $h\in H^0$ -  an external force, $f,g$  are subject to  the following conditions:
	\begin{itemize}
		\item $g\in C^1(\mathbb{R})$ is an increasing function such that $g(0)=0$, and there exist two positive constants $\kappa_0$ and $\kappa_1$ such that
		\begin{align}
			\label{hyp_g'}\kappa_0|s|^{4}\le g'(s)\le \kappa_1\left(1+|s|^{4}\right),\quad \forall s\in \mathbb{R}.
		\end{align}
		\item  $f\in C^2 (\mathbb{R})$, $f(0)=0$ and  there exists a positive constant $L_f>0$ such that
		\begin{eqnarray}
			\label{hyp_f''}\left|f''(s)\right|\le L_f\left(1+|s|^{3}\right),\quad \forall s\in \mathbb{R}.
		\end{eqnarray}
		The following standard  dissipative condition is imposed.
		\begin{align}\label{hyp-inf-f}\lim_{|s|\rightarrow \infty}\frac{f(s)}{s}>-\lambda_1\end{align}
		where $\lambda_1$ is the first eigenvalue of the Laplace operator  $-\Delta$ with Dirichlet boundary data.
	\end{itemize}
\end{Assumption}

\begin{remark} From Mean Value Theorem and \eqref{hyp_f''}, we get
	\begin{eqnarray}\label{hyp_f'}
		|f'(s)|\le C_f(1+|s|^4),\quad \forall s\in \mathbb{R},
	\end{eqnarray}
	for some constant $C_{f}>0$.
	Condition \eqref{hyp-inf-f} implies that there exists a constant $\nu\in [0,\lambda_1)$ and $C_{\nu}>0$ such that $F(s):=\int_0^sf(\tau)d\tau$ satisfies:
	\begin{eqnarray}
		\label{hyp_f2}-C_{\nu}-\frac{\nu}{2}|s|^2\le F(s)\le f(s)s+\frac{\nu}{2}|s|^2,\quad \forall s\in \mathbb{R},
	\end{eqnarray}
	
\end{remark}
Throughout the article we will use $0<\omega\le 1$ as the constant defined by
\begin{align*}
	\omega:=1-\frac{\nu}{\lambda_1}>0.
\end{align*}
Both sources and the dissipation are referred  to  as  critical in 3 dimensions. The criticality of the Sobolev's embedding $H^1\hookrightarrow L^6(\Omega)$ implies the loss of compactness of the potential energy $\int_{\Omega}F(u(x))dx$  along  with the loss of compactness of the damping $g(u_t) $ from  $H^1\rightarrow H^{-1} $.

\subsection{Statements of the  Main Results}
Of main interest to this paper are the so called \enquote{weak solutions}. There are several notions of \enquote{weak solutions} which are often refereed to as \enquote{finite energy solutions}. However, in the critical case one needs to be careful regarding precise meaning of weak solutions. While proving existence of strong [regular solutions] is rather routine, existence and the properties of weak solutions are more elusive in the critical case. Typical notion would be the so called \enquote{generalized} solutions which are defined as strong limits of semigroup regular solutions. Under certain conditions, such solutions can be shown to satisfy variational form of equations-the so called \enquote{weak solutions}. Weak solutions are not necessarily limits of strong solutions [unless some compactness related properties are assumed]- so these are of interest to the study in \enquote{critical} situations. Here, by critical we refer to the property that potential energy is just $L^1$ integrable without any compactness properties assumed in prior literature \cite{fereisl,vitillaro} and references therein. In what follows, we shall adopt rather classical definition -following \cite[Definition 2.1]{Lorena-lasiecka} where solutions to (\ref{P}) under the Assumption \ref{Assumption} were treated.

\begin{definition} \label{weak-solution} Let $Q:=\Omega\times (0,T)$. For any $T>0$, a function $U(t)=(u(t),u_t(t))\in C_w(0,T;\mathcal{H})$, such that $u_t\in L^6(Q)$ and possessing the property $U(0)=(u(0),u_t(0))=(u_0,u_1)\in \mathcal{H}$ is said to be a {\bf weak energy solution} to problem \eqref{P} on interval $[0,T]$ iff Eq. \eqref{P} is satisfied in the following variational sense:
	\begin{align}
		\label{variational-problem}
		\int_0^T\int_{\Omega}
		\left[
		-u_t\psi_t
		+\nabla u\cdot\nabla\psi
		+g(u_t)\psi
		+f(u)\psi
		\right]\,dxdt
		=
		-\left.\int_{\Omega}u_t\psi\,dx\right|_0^T
		+\int_0^T\int_{\Omega}h\psi\,dxdt,
	\end{align}
	for every
	$
	\psi \in C\bigl([0,T];H^1\bigr)\cap C^1\bigl([0,T];H^0\bigr)\cap L^6(Q).
	$
\end{definition}
Here $Q:=\Omega\times (0,T)$ and $C_w(0,T,Y)$ denotes the space of weakly continuous functions with values in a Banach space $Y$.

The first result refers to an existence and properties of weak solutions-defined in  Definition \ref{weak-solution}.
\begin{theorem} \label{theo-global}
	In reference to problem \eqref{P}, under Assumption \ref{Assumption}, the following statements hold.
	\begin{itemize}
		\item[\bf (I)] {\bf Weak solutions:} If initial data $U_0\in \mathcal{H}$, then problem \eqref{P} has a weak solution $U=(u,u_t)$ in the sense of Definition \ref{weak-solution}.
		In addition $U(\cdot) \in C(0,T;\mathcal{H}),  u_{tt} \in L^{6/5}(0,T; H^{-1}).$
		\item[\bf (II)] {\bf Energy equality:} The weak solutions satisfy the energy equality
		\begin{align}\label{energy-equality}E(U(t))+\int_s^t\int_{\Omega}g(u_t)u_tdxds=E(U(s)),\quad \mbox{for all}\quad t\ge s\ge0,\end{align}
		where $E(U)$ is the energy functional associated with problem \eqref{P} given by the formula
		\begin{align}
			\label{functional-energy}E(U(t)):=\frac{1}{2}||U(t)||^2_{\mathcal{H}}+\int_{\Omega}F(u)dx-\int_{\Omega}hudx,
		\end{align}
		\item[\bf(III)] {\bf Hadamard well-posedness :} Let $U^1=(u^1,u^1_t)$ and $U^2=(u^2,u^2_t)$ be two weak solutions of problem \eqref{P} with $U^i(0)=(u^i_0,u^i_1)=U^i_0\in \mathcal{H}$, $i=1,2$. Then
		\begin{align}\label{Lipschitz-property}
			||U^1(t)-U^2(t)||_{\mathcal{H}}\le e^{Ct}||U^1_0-U^2_0||_{\mathcal{H}},\quad t>0,
		\end{align}
		where the constant $C=C(||U^{i}_0||_{\mathcal{H}})>0$, $i=1,2$. In particular, if $U^1_0=U^2_0$ the  only solution is zero solution, so  weak  solutions are unique.
	\end{itemize}
\end{theorem}
\begin{remark}
	Since weak solutions are constructed via Galerkin approximations, the uniqueness statement  in Theorem \ref{theo-global}  applied to weak solutions implies  that all weak solutions can be constructed as Galerkin's approximations.
\end{remark}
The proof of the well-posedness theorem is based on the method introduced in \cite{Lorena-lasiecka}--which takes advantage of the enhanced dissipation exhibited by the additional integrability of the velocity. Existence of solutions was obtained in \cite{Bociu-lasiecka-jde} by constructing suitable approximations of critical nonlinearities followed by the use of monotone operator theory. Uniqueness was shown in \cite{Lorena-lasiecka} by first establishing the energy identity. However, continuous dependence stated in \cite{Lorena-lasiecka} holds {\it for smoother} initial data only. Since our ultimate goal is to have a properly defined dynamical system, we shall revisit the proof by using a different method based on Galerkin approximations. We will be able to establish full Hadamard well-posedness valid for all weak solutions in the {\it energy critical} case. This is the main step toward the construction of the dynamical system $(\cH,S_t)$, whose existence is guaranteed by Theorem \ref{theo-global}. In addition, we shall show that all weak solutions can be obtained as strong limits of Galerkin solutions. This is in line with \cite{Kalantarov-Savostianov-Zelik} and allows one to use Galerkin approximations in estimates where there is not enough smoothness to use operational calculus and classical limiting arguments.

Models with nonlinear subquintic terms and enhanced dissipation were considered earlier in \cite{fereisl,vitillaro}. However, in the case of weak solutions, {\it uniqueness} of solutions and Hadamard well-posedness were considered open problems. This issue has been positively resolved in \cite{Lorena-lasiecka} and in the present paper.

The main results of this work describe the long-time behavior of solutions generated by the dynamical system $(\mathcal{H},S_t)$. To maintain the text concise and avoid overloading it with preliminaries, the definitions and results related to the abstract theory of dynamical systems--particularly those concerning global attractors and their properties--will be referenced to the results established in standard sources. For the sake of concreteness, we shall provide precise references to \cite{chueshov-white,chueshov-yellow}.

\begin{theorem}\label{theo_main1}
	{\bf [Global attractor]} Assume that the hypotheses of Theorem \ref{theo-global} hold. Then, the associated dynamical system $(\mathcal{H},S_t)$ corresponding to the dynamics in \eqref{P} has a compact global attractor $\mathfrak{A}$ of the form $\mathfrak{A}=\mathrm{M}^u(\mathcal{N})$, where $\mathrm{M}^u(\mathcal{N})$ is the unstable manifold emanating from the set of stationary solution $\mathcal{N}$. Moreover, $\mathfrak{A}$ consists of full trajectories $\gamma=\{S_tU_0=U(t):t\in \mathbb{R}\}$ such that 
	\begin{align*}\lim_{t\rightarrow -\infty}\mbox{dist}_{\mathcal{H}}(U(t),\mathcal{N})=0\quad\mbox{and}\quad \lim_{t\rightarrow +\infty}\mbox{dist}_{\mathcal{H}}(U(t),\mathcal{N})=0.\end{align*}
\end{theorem}
The proof of existence of the global attractor is based on a generalization of \enquote{Ball's} method that employs energy arguments and is fundamentally based on the validity of the {\it energy equality} valid for weak solutions. This method is adapted to the present context by first establishing the validity of the {\it energy equality} and by incorporating critical nonlinear dissipation. The latter is possible due to the monotonicity properties.\\
Under the additional assumption imposed on the damping $g'(0)>0$, the following properties of the attractor $\mathfrak{A}$ are valid.
\begin{theorem}\label{theo-quasi} {\bf [Quasi-stability property.]}
	Let Assumption \eqref{hyp_g'} hold and $g'(0)>0$. Then, the dynamical system $(\mathcal{H},S_t)$ is quasi-stable.
\end{theorem}

The proof that the dynamical system is quasi-stable follows from the fact that the dynamical system $(\mathcal{H},S_t)$ satisfies the so called quasi-stability inequality, which is an enhanced version of the squeezing property -see \cite[Definition 7.9.2]{chueshov-yellow}.

\begin{theorem}\label{theo_main2}{\bf[Properties of the attractor]}
	
	Under the assumptions of Theorem \ref{theo-quasi}, the following properties of the attractor $\mathfrak{A}$ are valid.
	
	\begin{itemize}
		\item[{\bf(I)}]  {\bf Finite-dimensionality:}
		the global attractor $\mathfrak{A}$ has finite fractal dimension  $dim^\mathcal{H}_f \mathfrak{A}$;
		\item[{\bf(II)}] 
		{\bf Regularity:} any full trajectory $\gamma=\{U(t):t\in \mathbb{R}\}$ from attractor $\mathfrak{A}$ enjoys the following regularity properties
		$$U=(u,u_t)\in L^{\infty}(\mathbb{R};\mathcal{H}_1=H^2\cap H^1\times H^1),\quad u_{tt}\in L^{\infty}(\mathbb{R};H^0).$$
		Moreover, there exists $R_*>0$ such that:
		$\|u_{tt}(t)\|^2+\|u_t(t)\|^2_{H^1}+\|u(t)\|^2_{H^2}
		\le R^2_*,$
		for all $t\in \mathbb{R}$ and some positive constant $R_*$.
		\item[\bf(III)] {\bf Exponential attractor:} the dynamical system $(\mathcal{H},S_t)$ possesses a generalized fractal exponential attractor $\mathfrak{A}_{\exp}$ with finite dimension in the extended space $\mathcal{H}_{-s}:=H^{-s+1}\times H^{-s}$ with $0<s\le 1$.
	\end{itemize}
\end{theorem}	
The proof of Theorem \ref{theo_main2} is based on the validity of the quasi-stability property of the system $(\mathcal{H},S_t)$. This property is related to an exponential rate of approximation of sets that attract each other, therefore, the assumption that $g'(0)>0$ is expected.

The result presented above are inspired by earlier papers \cite{Lorena-lasiecka,Bociu-lasiecka-jde} where the model with critical [quintic] sources and enhanced dissipations were considered. For these models an existence and uniqueness of weak solutions has been established in references cited above. Hadamard well-posedness has been shown in \cite{Lorena-lasiecka} for {\it subcritical} sources only.
Questions related to blow up of weak solutions have been investigated in \cite{ram1,bociu1,ram3,ram2}.
Full Hadamard well-posedness and complete analysis of long time dynamics [including finite-dimensionality of the attractor] is the main contribution of the present manuscript. This includes the final assertion that {\it quintic waves can be stabilized [uniformly/exponentially] to a finite-dimensional set and characterised by finitely many degrees of freedom.} This opens the door to a construction of determining functionals \cite{chueshov-white,ladyzenskaya} and data assimilation methods which provide a basis for computational analysis. 

We conclude this section by briefly outlining general approach taken in the paper. The existence of weak solutions is established by Galerkin method with suitable $L^p$ global estimates. Weak solutions are shown to satisfy {\it energy equality}. Here, again, enhanced dissipation plays a major role. This is a critical step in order to prove uniqueness of weak solutions and Hadamard well-posedness. The latter allows to construct a well-posed dynamical system. By resorting, again, to enhanced dissipation, the system $(\mathcal{H},S_t)$ is shown to be ultimately dissipative, hence admitting weak global attractor. The next critical step is asymptotic compactness. Relying on already proved energy identity, an adaptation of the method of \cite{ball} to account for nonlinear damping, allows to establish asymptotic compactness, hence an existence of global attractor.
Under the additional assumption of non-degeneracy of the damping at the origin, global attractor is shown to be: smooth and also finite-dimensional. In addition, the dynamical system has also exponential attractor. These last properties follow from the proven {\it quasi-stability inequality}.
\subsection{Past contributions-history of the problem.}
Wave dynamics in 3 dimensions with subquintic and quintic source has been an object of long-time and intense activities. The main issues of interest to this work are well-posedness and long time behavior. The presence of the quintic source makes the traditional methodology inapplicable-due to the lack of appropriate embeddings which would \enquote{close} the needed estimates. Relatively recent development of Strichartz estimates opened the door to novel methodologies, where the additional space integrability is achieved due to compensation with time integrability. Careful balance of time-space integrability allows to construct local solutions referred to as Shatah-Struwe [S-S] solutions. However, such solutions do not admit global  [in time] extensions [unless the case is subcritical]. Here, the rescue comes by appealing to non-concentration argument \cite{Burq-Lebeau}. This argument relies on finite speed of propagation, hence sensitive to nonlinear damping. And this is one of the main points where nonlinearity of the damping in the present model becomes predicament. In fact, in the case of {\it linearly} bounded dissipation, the methodology pursued in \cite{Kalantarov-Savostianov-Zelik} is effective. The starting point of the analysis in \cite{Kalantarov-Savostianov-Zelik} is local existence of S-S-solutions followed 
by a  non-concentration argument  \cite{Burq-Lebeau}. However, this argument,  obtained via contradiction, does not provide long time estimates.  As a consequence, the so called S-S solutions may be lost in the limiting process. In order to argue long time properties of the  solutions,  \enquote{trajectory dynamical system} \cite{zelik} is constructed and  a rather involved 
procedure  with  the so called backward trajectories is applied  \cite{Kalantarov-Savostianov-Zelik}. 

In the case of subquintic sources, the analysis is simpler [S-S solutions are global], though still technical. Some of the recent relevant references include \cite{liu,Yan} and references within, where nonlinear dissipation has been also treated. For the quintic waves with {\it nonlinear dissipation} some of the elements of the S-S method described above has been very recently applied to subquintic and quintic waves with nonlinear {\it enhanced} dissipation \cite{zhou}. In particular, backward trajectories method patterned after \cite{zelik} is the driving argument in establishing asymptotic compactness of the trajectories and their smoothness. This requires strong assumptions imposed on the damping [nondegeneracy of the slope at the origin] and also the source [coercivity of quintic estimate].   

One of the goals of this paper is to show that such restrictions, in general, are not needed. Moreover, in the nondegenerate case, quintic waves can be stabilised to finite dimensional sets. The model with enhanced dissipation takes its own life with no reference to either S-S solutions or backward trajectories. The main protagonist is the enhanced dissipation which could also be degenerate [$g'(0) =0$.] This alone provides asymptotic compactness effect for a large range of sources including the quintic-energy critical. No coercivity assumption is needed. The role of Strichartz estimates is taken, in some sense, by the enhanced dissipation generated by the nonlinearity of the damping. However, in the present case one is in a position to show that the enhanced dissipation is retained in the ultimate dissipativity estimate [unlike Strichartz estimate]. In view of this, the present work is a continuation of the model and methods considered \cite{Lorena-lasiecka} in the direction of dynamical systems and its long time behaviour.
The main player in this game is {\it energy equality} established for weak solutions - an earlier version of the equality can be found in \cite{Lorena-lasiecka}. Adaptation of J.Ball's method to nonlinear dissipation leads to the existence of global attractor. A general quasi-stability property for the dynamics is established under additional assumption of non-vanishing derivative of the damper. The latter implies {\it finite-dimensionality, enhanced regularity of the attractor and also an existence of exponential attractor}.
In order to recover maximal regularity of the attractor, the splitting method used by Chueshov and Lasiecka in \cite{chueshov-lasiecka-2007} in conjunction with gradient structure of the attractor and properties of the stationary set.
In conclusion, the results provided eliminate critical restrictions in \cite{zhou} and moreover lead to different methodology applicable to many other hyperbolic like dynamics with critical sources. The method is much simpler and it could also be applied to nonlocal damping under certain natural assumptions \cite{zhou,zhou1}.

One should also mention works where {\it strong damping} is added to the model. But this is a form of regularization where originally hyperbolic dynamics becomes parabolic [linearisation generates an analytic semigroup]. The latter is a very different problem. References \cite{C-C,C-T,G-M,K-Z}.\\
Quintic waves with {\it linear} localised dissipation defined for {\it stable} dynamics has been considered in \cite{Cavalcanti}, where exponential decays of the energy are claimed. Regular solutions for wave equation with super-critical [above quintic] exponents were considered in \cite{radu}. Finally, the case of a quintic source term coupled with nonlinear dissipation of the form $g(u_t)\sim |u_t|^{p}u_t$, with $0<p<4$, remains open. This intermediate regime lies between the limiting case $p=0$, treated by Kalantarov {\it et al.} in \cite{Kalantarov-Savostianov-Zelik}, and the case $p=4$ considered in the present work.

The remainder of this article is devoted to the proofs of the three theorems stated above. The definitions and results from the abstract theory of global attractors used herein may be found in the books \cite{chueshov-white,chueshov-yellow}.

\section{Proof of Theorem \ref{theo-global}: Weak Solutions and the Generation of the Dynamical System $(\mathcal{H},S_t)$}

The well-posedness of the energy-critical wave equation with enhanced nonlinear dissipation was investigated in \cite{Bociu-lasiecka-jde}. In that work, methods from monotone operator theory \cite{Barbu} were employed to construct semigroup solutions. However, several aspects of the theory developed in \cite{Bociu-lasiecka-jde}--such as robustness and continuous dependence on the initial data--
required additional assumptions. These limitations are removed in the present approach, which is based on the construction of Galerkin approximations together with suitable estimates exploiting the enhanced dissipation. The resulting bounds are subsequently used to rigorously justify a number of formal calculations throughout the paper.

{\it Here is  the plan of attack}.  Staring with {\it regular [strong]} solutions defined  via Gelerkin approximation with the values in  the domain of linearised generator, we construct  {\it generalised} solutions  by taking the limits on the phase space $\mathcal{H} $. The  generalised solutions are shown:  (i) to be strongly continuous in time with the values in $\mathcal{H}$; (ii) 
to satisfy the  integral  equality in line with Definition  2.1. 
The next step is  the establishing of the  {\it energy identity} valid for  weak solution. This is done via a special discretization method introduced by Koch and Lasiecka \cite{Koch-lasiecka}. The latter allows one to establish uniqueness of {\it weak }solutions.  This combined with the energy {\it identity} leads to Hadamard wellposedness.  However, 
we shall also prove this estimate directly  in a form of robust-continuous dependence on the  initial data  estimate.   This step also allows  to conclude that each weak solution can be obtained via a Galerkin limit. 

\subsection{Proof of Theorem \ref{theo-global}-{\bf(I)}: Construction of Weak Solutions}
The proof of existence of weak solutions is based on    approximation procedure. We begin by constructing regular solutions associated with sufficiently smooth initial data and then obtain weak solutions by a density argument together with strong convergence of these regular approximations. As in \cite{Lorena-lasiecka}, the enhanced dissipation mechanism plays a central role in deriving the a priori estimates needed for the analysis. Unlike \cite{Lorena-lasiecka}, where the construction is carried out through suitable truncations of the nonlinear terms combined with monotone operator techniques \cite{Barbu,lasiecka-tataru}, the present approach relies on nonlinear Galerkin approximations that generate regular solutions. This setting has the advantage that   it leads to the  explicit stability inequalities needed for the treatment of long time behavior. In addition, it lends itself to potential treatments within the context of numerical analysis.    The main difficulty remains the passage to the limit in the critical nonlinear terms [both dissipation and the source], which requires  compensated compactness and monotonicity arguments closely related to those developed in \cite{Lorena-lasiecka}  and exploiting enhanced dissipation. We also note that \cite{Lorena-lasiecka} includes boundary source terms in the critical range, whose treatment within the present Galerkin framework would require additional technical developments.

\subsubsection*{Galerkin's approximate solutions}
Let $0<\lambda_1\le \lambda_2\le \cdots$ be the eigenvalues of the Laplacian operator $-\Delta$ with Dirichlet boundary condition and $\omega_1,\omega_2,\cdots$ be the corresponding eigenfunctions such that they form an orthonormal basis in $H^0$. We assume the boundary $\Gamma=\partial \Omega$ regular enough such that $\omega_j \in H^2.$ Let $V_k=\mbox{Span}\{\omega_1,\cdots,\omega_k\}$
the subspace of $H^2$ generated by the first $k$ elements of $\{\omega_j\}_{j\in \mathbb{N}}$.
For $k\in \mathbb{N}$ we can construct a function $u^k$ given by
$
u^k(t)=\sum_{j=1}^k y_{jk}(t)\omega_j\in V_k,$ $  t\in[0,t_k),$
where $(y_{jk})$ is a local solution on $[0, t_k)\in[0, T)$ of the following system of ODEs:
\begin{eqnarray}\left\{\begin{array}{l}\label{approximate}\displaystyle{
			(u^k_{tt}(t),\omega_j) + (\nabla u^k(t),\nabla\omega_j)+(g(u^k_t(t)),\omega_j)+(f(u^k(t)),\omega_j)=(h,\omega_j)}\\
		\displaystyle{U^k(0)= (u^k(0),u^k_t(0))=(u_{0k},u_{1k})=U_{0k},\quad j=1,\cdots,k.}\end{array}\right.
\end{eqnarray}
In what follows, it is necessary to establish a priori estimates that allow the extension of the local solution to the entire interval $[0, T]$. With this estimate in hand, we can then take the limit in the approximate problem \eqref{approximate} to establish the existence of weak solutions to \eqref{P}.

\subsubsection*{A priori estimates}
\paragraph{\bf{\underline{The first estimate}}} 
We consider the approximate system \eqref{approximate} with initial data satisfying
\begin{equation}\label{initial-condition}
	U^k(0) \to U_0 \quad \text{strongly in } \mathcal{H}.
\end{equation}
Multiplying \eqref{approximate} by $y_{jk}'(t)$, summing with respect to $j$, and integrating in time from $0$ to $t$, we obtain the energy identity
\begin{align}\label{aprox_2}
	E(U^k(t))+\int_0^t\int_\Omega g(u_t^k)u_t^k\,dx\,ds=E(U^k(0)),
\end{align}
where the energy functional $E$ is defined in \eqref{functional-energy}.
Using the coercivity assumptions \eqref{hyp_g'} and \eqref{hyp_f2}, the growth conditions \eqref{hyp_g'} and \eqref{hyp_f'}, the embeddings $H^1\hookrightarrow L^{6}(\Omega)\hookrightarrow H^0$, and the convergence \eqref{initial-condition}, it follows from \eqref{aprox_2} that
\begin{align}\label{Estimate_I0}
	\|U^k(t)\|_{\mathcal H}^2+\int_0^t\|u_t^k(s)\|_6^6\,ds
	\le C,
	\qquad \forall\, t\in[0,t_k],
\end{align}
where $C=C(\|U_0\|_{\mathcal H})>0$.
Hence, estimate \eqref{Estimate_I0} allows the local approximate solution to be extended to the entire interval $[0,T]$, for every $T>0$. In particular, \eqref{Estimate_I0} yields
\begin{align}
	\label{bounded-1}
	U^k=(u^k,u_t^k)&\quad \text{is bounded in }L^\infty(0,T;\mathcal H),\\
	\label{bounded-2}
	u_t^k&\quad \text{is bounded in }L^6(Q).
\end{align}

\paragraph{\bf{\underline{The second estimate}}} 
Now, we consider the approximate problem \eqref{approximate} with initial data satisfying
\begin{equation}\label{conv-initia-data-H1}
	U^k(0) \to U_0 \quad \text{strongly in } \big(H^2\cap H^1\big)\times \big(H^1\cap L^{10}(\Omega)\big).
\end{equation}
Using the eigenvalue relation $-\Delta \omega_j = \lambda_j \omega_j$, we replace $\omega_j$ by $-\Delta \omega_j$ in \eqref{approximate}. Multiplying the resulting equation by $\lambda_jy'_{jk}(t)$ and summing over $j$, we deduce that
\begin{equation}\label{strongy-1}
	\begin{aligned}
		&\frac{d}{dt}\Bigg[
		\frac{1}{2}||U^k(t)||^2_{\mathcal{H}_1}
		+ \frac{1}{2}\int_{\Omega} f'(u^k)|\nabla u^k|^2\,dx
		+ \int_{\Omega} h\,\Delta u^k\,dx
		\Bigg]
		+ \int_{\Omega} g'(u_t^k)|\nabla u_t^k|^2\,dx \\
		&\quad= \frac{1}{2}\int_{\Omega} f''(u^k)\,u_t^k\,|\nabla u^k|^2\,dx.
	\end{aligned}
\end{equation}
From the dissipativity assumption \eqref{hyp-inf-f}, there exist constants $\mu \in [0,\lambda_1)$ and $N>0$ such that\begin{equation}\label{strongy-2}f'(s) \ge -\mu, \quad \text{for all } |s| > N.\end{equation}On the other hand, by \eqref{hyp_f'},\begin{equation}\label{strongy-3}|f'(s)| \le C_f(1+N^4), \quad \text{for all } |s| \le N.\end{equation} 
We decompose $\Omega = \Omega_1 \cup \Omega_2$, where\[\Omega_1 := \{x \in \Omega : |u^k(x)| > N\}\quad \mbox{and} \quad\Omega_2 := \{x \in \Omega : |u^k(x)| \le N\}.\]It follows from \eqref{strongy-2} and \eqref{strongy-3} that
\begin{align}\label{strong-4}\frac{1}{2}\int_{\Omega} f'(u^k)|\nabla u^k|^2\,dx&\ge -\frac{\mu}{2}\int_{\Omega_1} |\nabla u^k|^2\,dx- \frac{C_f(1+N^4)}{2}\int_{\Omega_2} |\nabla u^k|^2\,dx \ge -\frac{K_f}{2}\|\nabla u^k(t)\|^2,\end{align}
where $K_f := \max\{\mu,\, C_f(1+N^4)\}$. Moreover, by Hölder's and Young's inequalities,\begin{equation}\label{strong-5}\left|\int_{\Omega} h\,\Delta u^k\,dx\right|\le \|h\|^2 + \frac{1}{4}\|\Delta u^k(t)\|^2.\end{equation}
We introduce the perturbed energy functional\begin{eqnarray*}E_1(t) :=\frac{1}{2}||U(t)||^2_{\mathcal{H}_1}+ \frac{1}{2}\int_{\Omega} f'(u^k)|\nabla u^k|^2\,dx + \int_{\Omega} h\,\Delta u^k\,dx+ \|h\|^2+ \frac{K_f}{2}\|\nabla u^k(t)\|^2.\end{eqnarray*}
Combining \eqref{strong-4} and \eqref{strong-5}, we infer that
\begin{equation}\label{estimate-E_1}
	E_1(t) \ge \frac{1}{2}\|\nabla u^k_t(t)\|^2 + \frac{1}{4}\|\Delta u^k(t)\|^2\ge \frac{1}{4}||U^k(t)||^2_{\mathcal{H}_1}.
\end{equation}
Differentiating $E_1(t)$ and using \eqref{strongy-1}, we obtain
\begin{equation}\label{strongy-6}
	\frac{d}{dt}E_1(t)
	+ \int_{\Omega} g'(u_t^k)|\nabla u_t^k|^2\,dx
	= K_f \int_{\Omega} \nabla u^k \cdot \nabla u_t^k\,dx
	+ \frac{1}{2}\int_{\Omega} f''(u^k)\,u_t^k\,|\nabla u^k|^2\,dx.
\end{equation}
We now estimate the right-hand side of \eqref{strongy-6}. First, by Hölder's inequality, the embedding $H^1\hookrightarrow H^0$, and \eqref{estimate-E_1}, it follows that
\[
K_f \int_{\Omega} \nabla u^k \cdot \nabla u_t^k\,dx
\le \frac{K_f}{\lambda_1^{1/2}} \|\Delta u^k(t)\| \|\nabla u_t^k(t)\|
\le \frac{2\sqrt{2}K_f}{\lambda_1^{1/2}} E_1(t).
\]
Next, using Hölder's inequality with $\frac{1}{2}+\frac{1}{6}+\frac{1}{3}=1$, the embedding $H^1\hookrightarrow L^6(\Omega)$,  together with \eqref{Estimate_I0} and \eqref{estimate-E_1}, we obtain
\[
\frac{1}{2}\int_{\Omega} f''(u^k)\,u_t^k\,|\nabla u^k|^2\,dx
\le C\big(1+\|u^k(t)\|_6^3\big)\|u_t^k(t)\|_6 \|\nabla u^k(t)\|_6^2
\le C \|u_t^k(t)\|_6\, E_1(t).
\]
Substituting these estimates into \eqref{strongy-6}, we deduce
\begin{equation}\label{strongy-final}
	\frac{d}{dt}E_1(t)
	+ \int_{\Omega} g'(u_t^k)|\nabla u_t^k|^2\,dx
	\le C\big(1+\|u_t^k(t)\|^6_6\big)\, E_1(t).
\end{equation}
An application of Gronwall's inequality to \eqref{strongy-final} yields
\begin{equation}\label{strongy-gronwall}
	E_1(t)\le
	e^{\left(Ct+C\int_0^t \|u_t^k(s)\|_6^6\,ds\right)}\,E_1(0).
\end{equation}
Consequently, combining \eqref{estimate-E_1}, \eqref{conv-initia-data-H1}, and \eqref{Estimate_I0}, it follows from \eqref{strongy-gronwall} that
\begin{equation}\label{strongy-final-2}
	||U^k(t)||^2_{\mathcal{H}_1}
	\le C\big(T,\|\nabla u_1\|,\|\Delta u_0\|\big),
	\quad \forall t \in [0,T], \ \forall k \in \mathbb{N}.
\end{equation}

Therefore, \eqref{strongy-final-2} yields
\begin{equation}\label{bounded-3}
	U^k=(u^k,u^k_t) \quad \text{is bounded in } L^{\infty}\big(0,T; \mathcal{H}_1\big).
\end{equation}

\paragraph{\bf{\underline{The third estimate}}} 
From \eqref{approximate}, evaluating at $t=0$ and testing with $u^k_{tt}(0)$, we obtain
\begin{equation}\label{strong-7}
	\|u^k_{tt}(0)\|^2
	= (\Delta u_{0k},u^k_{tt}(0))-(g(u_{1k}),u^k_{tt}(0))-(f(u_{0k}),u^k_{tt}(0))+(h,u^k_{tt}(0)).
\end{equation}
Applying H\"older's inequality, assumptions \eqref{hyp_g'}--\eqref{hyp_f'}, and the Sobolev embedding $H^2\hookrightarrow L^{10}(\Omega)$, it follows from \eqref{strong-7} that
\begin{equation}\label{strong-8}
	\begin{aligned}
		\|u^k_{tt}(0)\|
		&\le \|\Delta u_{0k}\|+\|g(u_{1k})\|+\|f(u_{0k})\|+\|h\| \\
		&\le \|\Delta u_{0k}\|+C\big(1+\|u_{1k}\|_{10}^5\big)+C\big(1+\|\Delta u_{0k}\|^5\big)+\|h\|.
	\end{aligned}
\end{equation}
Hence, by \eqref{conv-initia-data-H1}, we deduce from \eqref{strong-8} that
\begin{align*}
	\|u^k_{tt}(0)\|\le C\big(\|\Delta u_0\|,\|u_{1k}\|_{10}\big).
\end{align*}

Next, differentiating \eqref{approximate} with respect to $t$, testing the resulting equation with $y''_{jk}(t)$, and summing over $j$, we obtain
\begin{equation}\label{strong-10}
	\frac{1}{2}\frac{d}{dt}||U^k_t(t)||^2_{\mathcal{H}}
	+ \int_{\Omega} g'(u^k_t)|u^k_{tt}|^2\,dx
	= -\int_{\Omega} f'(u^k)u^k_t u^k_{tt}\,dx.
\end{equation}
We estimate the right-hand side of \eqref{strong-10}. By Hölder's inequality with $\frac{1}{3}+\frac{1}{6}+\frac{1}{2}=1$, the embeddings $H^1\hookrightarrow L^6(\Omega)$ and $H^2\hookrightarrow L^{12}(\Omega)$, and estimate \eqref{strongy-final-2}, we have
\begin{align*}
	\left|\int_{\Omega} f'(u^k)u^k_t u^k_{tt}\,dx\right|
	\le C\big(1+\|u^k(t)\|_{12}^4\big)\|u^k_t(t)\|_6\|u^k_{tt}(t)\| \le C||U^k_t(t)||^2_{\mathcal{H}}.
\end{align*}
Substituting into \eqref{strong-10}, we deduce
\begin{equation}\label{strong-10-est}
	\frac{d}{dt}||U^k_t(t)||^2_{\mathcal{H}}
	+ \int_{\Omega} g'(u^k_t)|u^k_{tt}|^2\,dx
	\le C||U^k_t(t)||^2_{\mathcal{H}}.
\end{equation}
An application of Gronwall's inequality to \eqref{strong-10-est} yields
\begin{equation}\label{strong-10-gronwall}
	||U^k_t(t)||^2_{\mathcal{H}}
	\le e^{Ct}||U^k_t(0)||^2_{\mathcal{H}}.
\end{equation}
Finally, combining \eqref{strong-8}, \eqref{conv-initia-data-H1}, and \eqref{strong-10-gronwall}, we infer that
\begin{equation}\label{strong-11}
	||U^k_t(t)||^2_{\mathcal{H}}
	\le C\big(\|\Delta u_0\|,\|u_{1k}\|_{10},\|\nabla u_{1k}\|\big),
	\quad \forall t\in[0,T],\ \forall k\in\mathbb{N}.
\end{equation}
Therefore, it follows from \eqref{strong-11} that
\begin{equation}\label{bounded-4}
	U^k_t=(u^k_t,u^k_{tt}) \quad \text{is bounded in } L^{\infty}(0,T;\mathcal{H}).
\end{equation}

\subsubsection*{Limit process}
We begin by establishing the existence of solutions with higher regularity, which will serve as the approximation class in the construction of weak solutions. More precisely, for sufficiently regular initial data $
U_0 \in (H^2\cap H^1)\times (H^1\cap L^{10}(\Omega)),
$
we show that problem \eqref{P} admits a regular solution satisfying the properties stated in \eqref{class-regular}. Once this is achieved, the existence of weak solutions to problem \eqref{P} will follow by  passing to the limit from this class of regular solutions and also using density arguments.
\begin{lemma}{\bf [Regular solutions]}
	Suppose that Assumption \ref{Assumption} holds. If
	\[
	U_0 \in (H^2\cap H^1)\times (H^1\cap L^{10}(\Omega)),
	\]
	then problem \eqref{P} admits a regular solution $u$ such that
	\begin{align}\label{class-regular}
		U=(u,u_t)\in C(0,T;\mathcal{H}_1), 
		\quad u_{tt}\in L^{\infty}(0,T;H^0), 
		\quad u_t\in L^6(Q).
	\end{align}
\end{lemma}
\begin{proof}
	From the uniform bounds \eqref{bounded-1}, \eqref{bounded-2}, \eqref{bounded-3}, and \eqref{bounded-4}, we may extract a subsequence (still denoted by $\{u^k\}$) such that
	\begin{equation}\label{cov-01}
		\left\{
		\begin{aligned}
			U^k &\rightharpoonup U=(u,u_t) 
			&&\text{weakly* in } L^{\infty}\big(0,T;\mathcal{H}_1\big),\\
			U_t^k &\rightharpoonup U_t=(u_t,u_{tt})
			&&\text{weakly* in } L^{\infty}(0,T;\mathcal H),\\
			u_t^k &\rightharpoonup u_t
			&&\text{weakly in } L^6(Q),\\
			g(u_t^k)&\rightharpoonup \kappa
			&&\text{weakly in } L^{6/5}(Q).
		\end{aligned}
		\right.
	\end{equation}
	By Lions' compactness theorem \cite[Ch.~1, Th.~5.1]{Lions}, together with the compact embedding $
	\mathcal{H}_1 \hookrightarrow \mathcal H,
	$
	we also obtain
	\begin{equation}\label{cov-2}
		U^k\to U
		\quad \text{strongly in } L^{\infty}(0,T;\mathcal H).
	\end{equation}
	In particular, $U\in C([0,T];\mathcal H)$.
	These convergences allow us to pass to the limit in the approximate equation \eqref{approximate}, yielding a function $u$ satisfying \eqref{class-regular} and
	\[
	u_{tt}-\Delta u+\kappa+f(u)=h
	\quad \text{in } L^{\infty}(0,T;H^0).
	\]
	Using the boundedness of $(u^k)$ in $L^\infty(0,T;H^2)$ together with the Sobolev embedding $
	H^2\hookrightarrow L^\infty(\Omega),
	$ we infer that $(u^k)$ is bounded in $L^\infty(Q)$. Consequently, by the growth assumption on $f$, the sequence $(f(u^k))$ is bounded in $L^\infty(0,T;H^0)$, which justifies the passage to the limit in the source term. The identification of $\kappa = g(u_t) $ is accomplished via "Strauss Lemma"  \cite{Brezis,Strauss} due to  (\ref{cov-01})  which, in particular, implies  [due to compactness of $H^1 \subset  H^0$]  strong convergence in $L^2(Q)$ . The latter yields   $a.e$ convergence in $Q$, as required in \cite{Strauss} . Finally, the initial conditions in \eqref{P} follow from \eqref{conv-initia-data-H1} and \eqref{cov-2}. Hence, problem \eqref{P} admits a regular solution.
\end{proof}
\begin{remark}\label{remark-1}
	\begin{enumerate}
		\item The previous lemma provides the regular approximation class used in the sequel to construct weak solutions by density and passage to the limit.
		\item For the regular solutions obtained above, the energy identity \eqref{energy-equality} established in Theorem \ref{theo-global} follows formally from the variational formulation \eqref{variational-problem} by taking the test function $\psi=u_t$. Moreover, uniqueness and continuous dependence follow by the same arguments used later for weak solutions in Section \eqref{cont-dependence}.
	\end{enumerate}
\end{remark}

\noindent\textbf{Existence of a weak solution.}
Let $U_0=(u_0,u_1)\in \mathcal{H}$. Since 
$
(H^2\cap H^1)\times (H^1\cap L^{10}(\Omega))
$
is dense in $\mathcal{H}$, there exists a sequence $\{U_0^m\}=(u_0^m,u_1^m)$ in this space such that
\begin{equation}\label{approx-init}
	U_0^m \to U_0 \quad \text{strongly in } \mathcal{H}.
\end{equation}
For each $m\in\mathbb{N}$, associated with the initial data $U_0^m$, there exists a {\it regular} solution $U^m=(u^m,u_t^m)$ satisfying the regularity stated in \eqref{class-regular} and solving
\begin{equation}\label{Pm}
	\begin{cases}
		u^m_{tt}-\Delta u^m+g(u_t^m)+f(u^m)= h, & \text{in } \Omega\times (0,\infty),\\
		u^m=0, & \text{on } \Gamma\times (0,\infty),\\
		u^m(\cdot,0)=u_0^m,\quad u_t^m(\cdot,0)=u_1^m, & \text{in } \Omega.
	\end{cases}
\end{equation}
Taking $u_t^m$ as a multiplier in \eqref{Pm}, arguing as in the first \emph{a priori} estimate, and using the convergence in \eqref{approx-init}, we obtain
\begin{align*}
	\|U^m(t)\|_{\mathcal H}^{2}
	+\int_{0}^{t}\|u_t^m(s)\|_{6}^{6}\,ds
	\le C,
	\qquad \forall\, t\in[0,T],
\end{align*}
for some constant $C=C(\|U_0\|_{\mathcal H})>0$.

Now, let $u^m$ and $u^n$ be two solutions of \eqref{Pm}, and set
$
z^{m,n}=u^m-u^n.
$
Then $z^{m,n}$ satisfies
\begin{equation}\label{diff-weak-solution-1}
	\begin{cases}
		z^{m,n}_{tt}-\Delta z^{m,n}
		+\bigl[g(u_t^m)-g(u_t^n)\bigr]
		+\bigl[f(u^m)-f(u^n)\bigr]=0,
		& \text{in } \Omega\times(0,\infty),\\[1mm]
		z^{m,n}=0,
		& \text{on } \Gamma\times(0,\infty),\\[1mm]
		z^{m,n}(\cdot,0)=u_0^m-u_0^n,\qquad
		z_t^{m,n}(\cdot,0)=u_1^m-u_1^n,
		& \text{in } \Omega.
	\end{cases}
\end{equation}
Taking $z^{m,n}_t$ as a multiplier in \eqref{diff-weak-solution-1}, we obtain
\begin{equation}
	\label{conv-g-a-1}
	\begin{aligned}
		&\frac{1}{2}\|U^m(t)-U^n(t)\|^2_{\mathcal{H}}
		+\int_0^t\int_{\Omega}\left[g(u_t^m)-g(u_t^n)\right]z^{m,n}_t\,dx\,ds\\
		&\quad+\int_0^t\int_{\Omega}\left[f(u^m)-f(u^n)\right]z^{m,n}_t\,dx\,ds=\;\frac{1}{2}\|U^m(0)-U^n(0)\|^2_{\mathcal{H}}.
	\end{aligned}
\end{equation}
Next, we write
\begin{equation}
	\label{conv-g-d-1}
	\begin{aligned}
		&\int_0^t\int_{\Omega}\left[f(u^m)-f(u^n)\right]z^{m,n}_t\,dx\,ds\\
		&=\int_0^t\left\{\frac{1}{2}\frac{d}{dt}\left[\int_{\Omega}\int_0^1 f'(\chi_\theta)\,d\theta\,|z^{m,n}|^2\,dx\right]-\frac{1}{2}\int_{\Omega}\int_0^1 f''(\chi_\theta)(\chi_\theta)_t\,d\theta\,|z^{m,n}|^2\,dx\right\}ds\\
		&=\frac{1}{2}\int_{\Omega}\int_0^1 f'(\chi_\theta(t))\,d\theta\,|z^{m,n}(t)|^2\,dx-\frac{1}{2}\int_{\Omega}\int_0^1 f'(\chi_\theta(0))\,d\theta\,|z^{m,n}(0)|^2\,dx\\
		&\quad-\frac{1}{2}\int_0^t\int_{\Omega}\int_0^1 f''(\chi_\theta)(\chi_\theta)_t\,d\theta\,|z^{m,n}|^2\,dx\,ds
	\end{aligned}
\end{equation}
where $\chi_\theta:=\theta u^m+(1-\theta)u^n$. Then, substituting \eqref{conv-g-d-1} into \eqref{conv-g-a-1}, we obtain
\begin{equation}
	\label{conv-g-aa}
	\begin{aligned}
		&\frac{1}{2}\|U^m(t)-U^n(t)\|^2_{\mathcal{H}}
		+\frac{1}{2}\int_{\Omega}\int_0^1 f'(\chi_\theta)\,d\theta\,|z^{m,n}|^2\,dx+\int_0^t\int_{\Omega}\left[g(u_t^m)-g(u_t^n)\right]z^{m,n}_t\,dx\,ds \\
		&=\frac{1}{2}\|U^m(0)-U^n(0)\|^2_{\mathcal{H}}
		+\frac{1}{2}\int_{\Omega}\int_0^1 f'(\chi_\theta(0))\,d\theta\,|z^{m,n}(0)|^2\,dx\\
		&\quad+ \frac{1}{2}\int_0^t\int_{\Omega}\int_0^1 f''(\chi_\theta)(\chi_\theta)_t\,d\theta\,|z^{m,n}|^2\,dx\,ds.
	\end{aligned}
\end{equation}

Define
\[
E_{z^{m,n}}(t):=
\frac{1}{2}\|U^m(t)-U^n(t)\|_{\mathcal H}^{2}
+\frac{1}{2}\int_{\Omega}\int_{0}^{1}f'(\chi_\theta)\,d\theta\,|z^{m,n}|^{2}\,dx
+\frac{K_f}{2}\|z^{m,n}(t)\|^{2},
\]
where \(K_f>0\) is a constant arising from the dissipativity assumption
\eqref{hyp-inf-f}, obtained by arguments analogous to those used in
\eqref{strong-4}. With a slight abuse of notation, we keep the same symbol
\(K_f\) for this constant. It is chosen in such a way that
\begin{align}\label{rel_E}
	E_{z^{m,n}}(t)\ge
	\frac{1}{2}\|U^m(t)-U^n(t)\|_{\mathcal H}^{2}.
\end{align}
Using the definition of $E_{z^{m,n}}(t)$, we can rewrite \eqref{conv-g-aa} as follows:
\begin{equation}\label{2gn-1}
	\begin{aligned}
		E_{z^{m,n}}(t)
		+\int_0^t\int_{\Omega}\left[g(u_t^m)-g(u_t^n)\right]z^{m,n}_t\,dx\,ds
		=E_{z^{m,n}}(0)+ I_1 + I_2,
	\end{aligned}
\end{equation}
where
\[
I_1=K_f\int_0^t\int_{\Omega} z^{m,n} z^{m,n}_t\,dx\,ds,\qquad
I_2=\frac{1}{2}\int_0^t\int_{\Omega}\int_0^1 f''(\chi_\theta)(\chi_\theta)_t\,d\theta\,|z^{m,n}|^2\,dx\,ds.
\]
On the one hand, using the embedding $H^1\hookrightarrow H^0$ e \eqref{rel_E}, it is straightforward to see that
\begin{align*}
	|I_1| \le \frac{K_f}{2\lambda_1^{1/2}}\int_0^tE_{z^{m,n}}(s)\,ds.
\end{align*}
On the other hand, by Hölder's inequality with exponents $\frac{1}{2}+\frac{1}{6}+\frac{1}{3}=1$ and the embedding $H^1\hookrightarrow L^6(\Omega)$, we obtain
\begin{eqnarray*}
	|I_2|
	&\le& C\int_0^t\big(1+\|u^m(s)\|_6^3+\|u^n(s)\|_6^3\big)
	\big(\|u_t^m(s)\|_6+\|u_t^n(s)\|_6\big)\|z^{m,n}(s)\|_6^2\,ds\\ 
	&\le& C\int_0^t\psi(s,u^m,u^n)\,E_{z^{m,n}}(s)ds,
\end{eqnarray*}
where $\psi(t,u^m,u^n):=1+\|u_t^m(t)\|_6^6+\|u_t^n(t)\|_6^6 \in L^1(0,T).
$ 

Substituting the above estimates for $|I_1|$ and $|I_2|$ into \eqref{2gn-1}, we obtain
\begin{align*}
	E_{z^{m,n}}(t)
	+\int_0^t\int_{\Omega}\left[g(u_t^m)-g(u_t^n)\right]z^{m,n}_t\,dx\,ds
	\le E_{z^{m,n}}(0)+ C\int_0^t\psi(s,u^m,u^n)\,E_{z^{m,n}}(s)ds.
\end{align*}
Hence, by Gronwall's inequality,
\begin{eqnarray}\label{2gn3-1}
	E_{z^{m,n}}(t)
	+\int_0^t\int_{\Omega}\left[g(u_t^m)-g(u_t^n)\right]z_t^{m,n}\,dx\,ds
	\leq
	E_{z^{m,n}}(0)e^{\left(C\int_0^t\psi(s,u^m,u^n)\,ds\right)}.
\end{eqnarray}
for all $t\in[0,T]$. By the monotonicity of $g$, the second term on the left-hand side is nonnegative. Moreover, since $\psi(\cdot,u^m,u^n)\in L^1(0,T)$, the exponential factor in \eqref{2gn3-1} is finite for all $t\in[0,T]$.
Using Assumption \eqref{hyp_f'}, H\"older's inequality with $\frac{2}{3}+\frac{1}{3}=1$, and the embedding $H^1\hookrightarrow L^6(\Omega)$, we estimate
\begin{align*}
	\left|\int_{\Omega}\int_0^1 f'(\chi_\theta)\,d\theta\,|z^{m,n}|^2\,dx\right|
	\le C\left(1+\|u^m(t)\|_6^4+\|u^n(t)\|_6^4\right)\|z^{m,n}(t)\|_6^2 \le C\|\nabla z^{m,n}(t)\|^2.
\end{align*}
Consequently, by the definition of $E_{z^{m,n}}$ and the embedding $H^1\hookrightarrow H^0$, we have
\begin{align}\label{rel_E-2}
	E_{z^{m,n}}(t)\le C\|U^m(t)-U^n(t)\|^2_{\mathcal{H}}.
\end{align}
Finally, combining \eqref{rel_E} and \eqref{rel_E-2} [evaluated at $t=0$] and using \eqref{2gn3-1}, we deduce that
\begin{equation}\label{2gn5-1}
	\|U^m(t)-U^n(t)\|^2_{\mathcal{H}}
	\le C\big(T,\|U^m_0\|_{\mathcal{H}},\|U^n_0\|_{\mathcal{H}}\big)\,
	\|U^m_0-U^n_0\|^2_{\mathcal{H}},
	\quad \forall t\in[0,T].
\end{equation}
In view of \eqref{2gn5-1} and the convergence of the initial data in \eqref{approx-init}, it follows that the sequence $\{U^m\}$ is Cauchy in $C([0,T];\mathcal{H})$ and therefore convergent. Consequently, there exists a function $U=(u,u_t)\in C([0,T],\mathcal{H})$ such that
\begin{align}\label{cov-1}
	\lim_{m\to \infty}\max_{t\in [0,T]}||U^m(t)-U(t)||_{\mathcal{H}} =0.
\end{align}
\begin{remark}
	The function $U$ satisfying \eqref{cov-1} above, as defined in \cite[Definition 1.4]{chueshov-white}, is called a {\bf generalized solution} of problem \eqref{P}. We next show that such solutions satisfy the variational formulation \eqref{variational-problem} of problem \eqref{P} and, consequently, are weak solutions in the sense of Definition \ref{weak-solution}.
\end{remark}

Now, we proceed taking the limit in the approximate problem \eqref{approximate}, thereby demonstrating that a limiting solution satisfies the variational problem \eqref{variational-problem}. As noted in Remark \ref{remark-1}, passing to the limit in the linear terms is straightforward. Therefore, we focus on establishing the limit passage for the critical nonlinear terms $f$ and $g$. This is carried out as follows.

\paragraph{\underline{Analysis of the source term $f(u^m)$}}
Using \eqref{cov-1}, by continuity of $f$ it follows that
\begin{align}\label{conv-u-a.e.}
	u^m \to u \quad \mbox{a.e. in } Q,\qquad
	f(u^m)\to f(u)\quad \mbox{a.e. in } Q.
\end{align}
Now, from Assumption \eqref{hyp_f'}, the embeddings $H^1\hookrightarrow L^{6}(\Omega)\hookrightarrow L^{6/5}(\Omega)$, and \eqref{bounded-1}, we have
\begin{align*}
	\int_0^T\int_{\Omega}|f(u^m)|^{6/5}\,dx\,dt
	\le C\int_0^T\int_{\Omega}\big(|u^m|^{6/5}+|u^m|^{6}\big)\,dx\,dt<\infty.
\end{align*}
That is,
\begin{align}\label{bound-f}
	f(u^m)\quad \mbox{is bounded in}\quad L^{6/5}(Q).
\end{align}
Then, using \eqref{conv-u-a.e.} and \eqref{bound-f}, it follows from Lions' lemma \cite[Ch.~1, Lemma~1.3]{Lions} that
\begin{align}\label{XXX}
	f(u^m)\rightharpoonup f(u)\quad \mbox{weakly in}\quad L^{6/5}(Q)=[L^6(Q)]'.
\end{align}
Thus, from \eqref{XXX}, we have
\[
\int_0^T (f(u^m),\psi)\,dt \to \int_0^T (f(u),\psi)\,dt,
\quad \text{as } m\to\infty,\quad \forall\, \psi \in L^6(Q).
\]

\paragraph{\underline{Analysis of the damping term $g(u^m_t)$}}
From Assumption \eqref{hyp_g'}, the embeddings $L^{6}(\Omega)\hookrightarrow H^0\hookrightarrow L^{6/5}(\Omega)$, and \eqref{bounded-2}, we have
\begin{align*}
	\int_0^T\int_{\Omega}|g(u^m_t)|^{6/5}\,dx\,dt
	\le C\int_0^T\int_{\Omega}\big(|u^m_t|^{6/5}+|u^m_t|^{6}\big)\,dx\,dt<\infty.
\end{align*}
That is,
\begin{align*}
	g(u^m_t)\quad \mbox{is bounded in}\quad L^{6/5}(Q).
\end{align*}
Consequently, we obtain
\begin{align*}
	g(u^m_t)\rightharpoonup \chi\quad \mbox{weakly in}\quad L^{6/5}(Q),
	\quad \mbox{for some } \chi\in L^{6/5}(Q).
\end{align*}
Passing to the limit as \(m\to\infty\) in the variational formulation associated with \eqref{approximate}, and using the above convergences, we infer that
\begin{align*}
	\int_0^T\int_{\Omega}
	\left[
	-u_t\psi_t
	+\nabla u\cdot\nabla\psi
	+\chi\psi
	+f(u)\psi
	\right]dxdt
	=
	-\left.\int_{\Omega}u_t\psi\,dx\right|_0^T
	+\int_0^T\int_{\Omega}h\psi\,dxdt,
\end{align*}
for every
$
\psi\in C([0,T];H^1)\cap C^1([0,T];H^0)\cap L^6(Q).
$

Thus, the limit function \(u\) satisfies the variational formulation of Definition~\ref{weak-solution}, with \(g(u_t)\) replaced by \(\chi\). It remains to show that $g(u_t)=\chi$.
To this end, we use the maximal monotonicity of the operator \(g\), together with \cite[Lemma~2.3]{Barbu}. According to this result, it suffices to show that
\begin{equation}\label{1g}
	\limsup_{m,n\to\infty}
	\int_0^T\int_{\Omega}
	\left[g(u_t^m)-g(u_t^n)\right]
	\left[u_t^m-u_t^n\right]dxdt
	\le 0.
\end{equation}
where the above inequality is expected to hold on regular solutions $u^m,u^n$. Indeed, from \eqref{2gn3-1}, we have
\begin{align}\label{2gn3}
	\underbrace{E_{z^{m,n}}(t)}_{\ge0}
	+
	\int_0^t\int_{\Omega}
	\left[g(u_t^m)-g(u_t^n)\right]
	z_t^{m,n}\,dxds
	\le
	e^{C\int_0^t\psi(s,u^m,u^n)ds}
	E_{z^{m,n}}(0),
\end{align}
for all \(t\in[0,T]\). Since \(\psi(\cdot,u^m,u^n)\in L^1(0,T)\), together with \eqref{approx-init}, it follows from \eqref{2gn3} that
\begin{align}\label{2gn4}
	\limsup_{m,n\to\infty}
	\int_0^t\int_{\Omega}
	\left[g(u_t^m)-g(u_t^n)\right]
	\left[u_t^m-u_t^n\right]dxds
	\le
	C\limsup_{m,n\to\infty}\|U_0^m-U_0^n\|_{\mathcal H}^2
	=0,
\end{align}
for all \(t\in[0,T]\). Hence, \eqref{2gn4} implies \eqref{1g}, and by \cite[Lemma~2.3]{Barbu}, we conclude that
\[
\chi = g(u_t).
\]
This completes the proof of Part ${\bf(I)}$ of Theorem \ref{theo-global}.

\ifdefined\xxxx
To prove this we will use the monotone operator method applied to nonlinear hyperbolic operators  established by Lions \cite[Ch. 2, Section 6]{Lions}. In order, using the multiplier $\omega=u^k_t$ in \eqref{approximate} and integral from $0$ to $T$ the obtained equation we find that
\begin{align}\label{Eq-xx}
	E(U^k(T))+\int_0^T\int_{\Omega}g(u^k_t)u^k_tdxdt=E(U^k(0)).
\end{align}
Using the definition of $E$ and the strong convergence of \eqref{initial-condition} on the initial data it is easy to verify that $E(U(0))=\lim_{k\rightarrow +\infty}E(U^k(0)).$
Moreover, from \eqref{equiv_E}, we have $E(U^k(T))\ge -C_{\nu}|\Omega|-\frac{1}{\omega\lambda_1}\|h\|^2.$ Then using convergences \eqref{conv-weak-star-U} and Fatou's lemma, we conclude that $E(U(t))\le \lim\inf_{k\rightarrow +\infty} E(U^k(T)).$
Hence, it follows from \eqref{Eq-xx} that
\begin{align}\label{Eq-xxx}
	E(U(T))+\lim\inf_{k\rightarrow +\infty}\int_0^T\int_{\Omega}g(u^k_t)u^k_tdxdt\le E(U(0)).
\end{align}
Next, writing \eqref{limit-eq-a.e} in the form
\begin{align}
	\label{limit-eq-a.e-2}
	u_{tt}-\Delta u=\underbrace{h}_{\in L^2(Q)}-\underbrace{(\chi+f(u))}_{\in L^{6/5(Q)}}
\end{align}
and applying \cite[Ch. 2, Lemma 6.1]{Lions} in \eqref{limit-eq-a.e-2}, we obtain that
\begin{align}
	\label{conv-z} \frac{1}{2}||U(t)||^2_{\mathcal{H}}\ge \frac{1}{2}||U(0)||^2_{\mathcal{H}}+\int_0^T\int_{\Omega}\left[h-\chi-f(u)\right]u_tdxdt.
\end{align}
Using that $f(u)\in L^{6/5}(Q)$, $h\in L^2(Q)$, and $u_t\in L^{6}(Q)\hookrightarrow L^{2}(Q)$, we have $f(u)u_t\in L^{1}(Q)$ and $hu_t\in L^{1}(Q)$. Then, we can approximate $u$ by smooth function and proceed with a standard argument to deduce that
\begin{align*}
	\int_{\Omega}\left[\,h-f(u)\,\right]u_tdx=\frac{d}{dt}\int_{\Omega}\left[\,hu-F(u)\,\right]dx,\quad \mbox{for almost}\quad t\in[0,T].
\end{align*}
From the above equality, it follows from \eqref{conv-z} that
\begin{align}
	\label{conv-zz} E(U(T))+\int_0^T\chi u_tdxdt\ge E(U(0)).
\end{align}
Thus, substituting \eqref{conv-zz} in \eqref{Eq-xxx}, we have
\begin{eqnarray}\label{Eq-xxxx}
	\lim\inf_{k\rightarrow +\infty}\int_0^T\int_{\Omega}g(u^k_t)u^k_tdxdt\le \int_0^T\int_{\Omega}\chi u_tdxdt.
\end{eqnarray}
From \eqref{Eq-xxxx}, we have
\begin{eqnarray}\label{conv-zzz}
	\lim\inf_{k\rightarrow +\infty}\int_0^T\int_{\Omega}\left[\,g(u^k_t)-g(v)\,\right][u^k_t-v]dxdt\le \int_0^T\int_{\Omega}\left[\chi-g(v)\right][u_t-v]dxdt,
\end{eqnarray}
for all $v\in L^6(Q)$. On the other hand, using \cite[Remark 4:1]{EECT-2017} with $\gamma=4$, we have
\begin{align}\label{conv-g-11}
	\int_0^T\int_{\Omega}\left[g(u^k_t)-g(v)\right][u^k_t-v]dxdt\ge \frac{\kappa_0}{5}\int_0^T\int_{\Omega}|u^k_t-v|^6dxdt\ge 0,\quad \forall v\in L^6(Q).
\end{align}
Thus, combining \eqref{conv-zzz} and \eqref{conv-g-11}, we conclude that
\begin{align}\label{conv-g-111}
	\int_0^T\int_{\Omega}\left[\,\chi-g(v)\right][u_t-v]dxdt\ge 0,\quad \forall v\in L^6(Q).
\end{align}
For any $w\in L^6(Q)$ and $\eta\in \mathbb{R}$, using $v=u_t-\eta w$ in \eqref{conv-g-111}, by standard arguments we obtain that
\begin{align*}
	\int_0^T\int_{\Omega}\left[\,\chi-g(u_t)\right]wdxdt= 0,\quad \forall w\in L^6(Q),
\end{align*}
which implies that  $\chi=g(u_t)$. \\fi
\fi

\begin{remark}
	We note that the same conclusion as above holds under weaker assumptions imposed on the source $f$. It suffices that the constraints on the lower bound of $f$ is imposed only for critical level of nonlinearity: $\lim_{ |s|\rightarrow \infty } \frac{f'(s)}{s^4}>-\nu $. The subcritical terms will converge to zero by the virtue of the  compactness.
\end{remark}

\subsection{Proof of Theorem \ref{theo-global}-{\bf(II)}: Energy Equality}	
The energy identity, as opposed to the weaker energy inequality, plays a central role in the well-posedness of the dynamical system generated by weak solutions. Indeed, weak solutions are typically obtained with only weak continuity in time as a consequence of a priori bounds, while the energy identity allows one to upgrade this property to strong continuity. In addition, the energy relation is a fundamental tool in the proof of asymptotic compactness, especially in the presence of critical nonlinearities, where standard methods based on squeezing estimates are often ineffective. In such cases, long-time trajectory estimates provide a more flexible approach; see \cite{ball}. The validity of the full energy identity is then essential. For the regular solutions constructed in the previous section, the energy identity follows formally by testing the equation with the velocity $u_t$. The main difficulty is to justify the same relation at the level of weak solutions, which are obtained as strong limits of regular solutions but do not possess sufficient regularity to support this argument directly. In the present model, the available integrability of $u_t$ does not allow one to use $u_t$ itself as an admissible test function in a straightforward way, since the weak formulation requires test functions in $H^1$. Although the energy inequality can be recovered from weak lower semicontinuity and approximation arguments, such methods do not yield the full identity. To overcome this difficulty, we follow the finite-difference approach introduced in \cite{Koch-lasiecka} and further developed in \cite[Appendix 5]{Assia-Irena}. This method provides a rigorous approximation of time derivatives and allows one to recover the exact energy balance for weak solutions. To proceed, let $v\in B([0,T];V)$, where $V$ is a Hilbert space and $B([0,T];V)$ denotes the space of $V$-valued bounded functions on $[0,T]$, endowed with the norm $
\|v\|_{B([0,T];V)}=\sup_{t\in[0,T]}\|v(t)\|_V.
$
Let $\epsilon>0$ and extend $v$ to all $t\in\mathbb R$ by setting $v(t)=v(T)$ for $t\ge T$ and $v(t)=v(0)$ for $t\le0$. We then define
$$
v_\epsilon^+(t)=v(t+\epsilon)-v(t),\qquad
v_\epsilon^-(t)=v(t)-v(t-\epsilon),\quad
\mbox{and}\quad
D_\epsilon v(t)=\frac{1}{2\epsilon}\left[\,v_\epsilon^+(t)+v_\epsilon^-(t)\right].
$$
The following proposition, stated in \cite[Proposition 1]{Lorena-lasiecka}, is a slight extension of related results proved in \cite[Proposition 5.2]{Assia-Irena} and \cite{Koch-lasiecka}.

\begin{proposition}
	\begin{itemize}
		\item Let $v$ be a weakly continuous with the values in a Hilbert space $V$. Then
		\begin{align}\label{energy-a}
			\lim_{\epsilon\rightarrow  0}\int_0^T\left(v(t),D_{\epsilon}v(t)\right)_Vdt=\frac{1}{2}\left[\,\|v(T)\|^2_V-\|v(0)\|^2_V\,\right].
		\end{align}
		\item For $v\in W^{1,p}(0,T;W)$, the following limits are well defined in $L^p(0,T;W)$, for all $1<p<\infty$ and any Banach space $W$:
		\begin{align}\label{energy-b}
			\lim_{\epsilon\rightarrow  0}D_{\epsilon}v=v_t,\quad\lim_{\epsilon\rightarrow  0}\frac{1}{\epsilon}v^+_{\epsilon}=v_t,\quad \lim_{\epsilon\rightarrow  0}\frac{1}{\epsilon}v^-_{\epsilon}=v_t.
		\end{align}
		Moreover, if $v_t$ is weakly continuous with the values in $W$, then for every $t\in(0,T)$,
		$D_{\epsilon}v(t)\rightarrow v_t(t)$ weakly in $W$
		and
		\begin{align*}
			\lim_{\epsilon\rightarrow  0}v^-_{\epsilon}(T)=v_t(T),\quad\lim_{\epsilon\rightarrow  0}\frac{1}{\epsilon}v^+_{\epsilon}(0)=v_t(0),\quad \mbox{weakly in}\quad W.
		\end{align*}
	\end{itemize}
\end{proposition}
Next, using variational form \eqref{variational-problem} of problem \eqref{P} with $\psi(t)=D_\epsilon u(t)\in H^1$, which is also continuous in time $C(0T;H^1) $  leads to:
\begin{eqnarray}
	\label{energy-equal-1}
	\begin{aligned}
		&-\int_0^T\int_{\Omega}u_t \frac{d}{dt} D_{\epsilon}u\,dxdt
		+\left.\int_{\Omega}u_tD_{\epsilon}u\,dx\right|_0^T
		+\int_0^T\int_{\Omega}\nabla u\cdot\nabla D_{\epsilon}u\,dxdt\\
		&\qquad
		+\int_0^T\int_{\Omega}g(u_t)D_{\epsilon}u\,dxdt
		+\int_0^T\int_{\Omega}f(u)D_{\epsilon}u\,dxdt
		=
		\int_0^T\int_{\Omega}hD_{\epsilon}u\,dxdt.
	\end{aligned}
\end{eqnarray}
To evaluate the integral on the left , we split the integration interval into three subintervals, namely
$[0,\epsilon]$, $[\epsilon,T-\epsilon]$, and $[T-\epsilon,T]$. Then, we perform suitable changes of variables in each corresponding integral and observe the cancellation of several terms. This leads to
\begin{eqnarray}
	\label{energy-equal-2a}
	\begin{aligned}
		&\;\int_0^T\int_{\Omega}u_t\frac{d}{dt}(D_{\epsilon}u)\,dxdt
		\\
		¨&=\;\frac{1}{2\epsilon}\left[\,\int_0^{\epsilon}\int_{\Omega}u_t(t)u_t(t+\epsilon)\,dxdt
		+\int_{\epsilon}^{T-\epsilon}\int_{\Omega}u_t(t)\bigl(u_t(t+\epsilon)-u_t(t-\epsilon)\bigr)\,dxdt\right.\\
		&\qquad\left.-\int_{T-\epsilon}^T\int_{\Omega}u_t(t)u_t(t-\epsilon)\,dxdt\,\right]\\
		&=\;\frac{1}{2\epsilon}\left[\,\int_0^{\epsilon}\int_{\Omega}u_t(t)u_t(t+\epsilon)\,dxdt
		+\int_{\epsilon}^{T-\epsilon}\int_{\Omega}u_t(t)u_t(t+\epsilon)\,dxdt\right.\\
		&\qquad\left.-\int_{0}^{T-2\epsilon}\int_{\Omega}u_t(t)u_t(t+\epsilon)\,dxdt
		-\int_{T-\epsilon}^T\int_{\Omega}u_t(t)u_t(t-\epsilon)\,dxdt\,\right]\\
		&=\;\frac{1}{2\epsilon}\left[\,\int_0^{\epsilon}\int_{\Omega}u_t(t)u_t(t+\epsilon)\,dxdt
		+\int_{T-2\epsilon}^{T-\epsilon}\int_{\Omega}u_t(t)u_t(t+\epsilon)\,dxdt\right.\\
		&\qquad\left.-\int_{0}^{\epsilon}\int_{\Omega}u_t(t)u_t(t+\epsilon)\,dxdt
		-\int_{T-\epsilon}^{T}\int_{\Omega}u_t(t)u_t(t-\epsilon)\,dxdt\,\right]=0.
	\end{aligned}
\end{eqnarray}
On the other hand, by using the definition of $D_{\epsilon}$ and the structure of the extension, one gets
\begin{eqnarray}
	\label{energy-equal-2b}
	\begin{aligned}
		\lim_{\epsilon\to 0}\int_{\Omega}u_t(T)D_{\epsilon}u(T)\,dx
		=&\;\lim_{\epsilon\to 0}\frac{1}{2\epsilon}\int_{\Omega}u_t(T)\bigl(u(T)-u(T-\epsilon)\bigr)\,dx
		=\frac{1}{2}\|u_t(T)\|^2,\\
		\lim_{\epsilon\to 0}\int_{\Omega}u_t(0)D_{\epsilon}u(0)\,dx
		=&\;\lim_{\epsilon\to 0}\frac{1}{2\epsilon}\int_{\Omega}u_t(0)\bigl(u(\epsilon)-u(0)\bigr)\,dx
		=\frac{1}{2}\|u_t(0)\|^2.
	\end{aligned}
\end{eqnarray}
Then, it follows from \eqref{energy-equal-2a} and \eqref{energy-equal-2b} that
\begin{align}
	\label{energy-equal-3}
	\lim_{\epsilon\to 0}
	\left\{
	-\int_0^T\int_{\Omega}u_t \frac{d}{dt}\left(D_{\epsilon}u\right)\,dxdt
	+\left.\int_{\Omega}u_tD_{\epsilon}u\,dx\right|_0^T
	\right\}
	=
	\frac{1}{2}\|u_t(T)\|^2-\frac{1}{2}\|u_t(0)\|^2.
\end{align}

\begin{remark} 
	Note that $D_{\epsilon} \frac{d}{dt }$ do not commute at the end intervals. 
\end{remark}

By direct application of \eqref{energy-a}, we get
\begin{align}
	\label{energy-equal-4}
	\lim_{\epsilon \rightarrow 0}\int_0^T\int_{\Omega}\nabla uD_{\epsilon}(\nabla u)dxdt=\frac{1}{2}\|\nabla u(T)\|^2-\frac{1}{2}\|\nabla u(0)\|^2.
\end{align}
Using Assumption \eqref{hyp_g'} and that $u_t\in L^{6}(Q)$, we have
\begin{align*}
	\int_0^T\int_{\Omega}|g(u_t)|^{6/5}dxdt\le C\int_0^T\int_{\Omega}[\,|u_t|^{6/5}+|u_t|^{6}\,]dxdt<\infty.
\end{align*}
That is, $g(u_t)\in L^{6/5}(Q).$ On the other hand, using $u_t\in L^6(Q)$, it follows from \eqref{energy-b} that $$\lim_{\epsilon \rightarrow 0}D_{\epsilon}(u)=u_t\in L^6(Q).$$
Hence, 
\begin{eqnarray}
	\label{energy-equal-5}
	\lim_{\epsilon\rightarrow 0}\int_0^T\int_{\Omega}g(u_t)D_{\epsilon}udxdt=\int_0^T\int_{\Omega}g(u_t)u_tdxdt.
\end{eqnarray}
Now, using Assumption \eqref{hyp_f'}, embedding $H^1\hookrightarrow L^{6}(\Omega)$ and that $u(t)\in H^1$ for all $t\in[0,T]$, we get
\begin{align*}
	\int_0^T\int_{\Omega}|f(u)|^{6/5}dxdt\le C\int_0^T\int_{\Omega}[\,|u|^{6/5}+|u|^{6}\,]dxdt<\infty.
\end{align*}
That is, $f(u)\in L^{6/5}(Q).$ Again using $\lim_{\epsilon \rightarrow 0}D_{\epsilon}(u(t))=u_t\in L^6(Q),$ we conclude that
\begin{eqnarray}
	\label{energy-equal-5'}
	\lim_{\epsilon\rightarrow 0}\int_0^T\int_{\Omega}f(u)D_{\epsilon}udxddt=\int_0^T\int_{\Omega}f(u)u_tdxddt.
\end{eqnarray}
Finally, it follows directly from \eqref{energy-b} that
\begin{eqnarray}
	\label{energy-equal-6}
	\lim_{\epsilon\rightarrow 0}\int_0^T\int_{\Omega}hD_{\epsilon}udxddt=\int_0^T\int_{\Omega}hu_tdxddt.
\end{eqnarray}
Thus, taking the limit as $\epsilon \rightarrow 0$ in \eqref{energy-equal-1} and using convergences \eqref{energy-equal-3}-\eqref{energy-equal-6}, we conclude that
\begin{eqnarray}
	\label{energy-equal-7}
	\begin{aligned}
		&\frac{1}{2}\left[\,\|u_t(T)\|^2-\|u_t(0)\|^2\,\right]+\frac{1}{2}\left[\,\|\nabla u(T)\|^2-\|\nabla u(0)\|^2\,\right]\\
		&+\int_0^T\int_{\Omega}g(u_t)u_tdxddt+\int_0^T\int_{\Omega}f(u)u_tdxdt=\int_0^T\int_{\Omega}hu_tdxdt.
	\end{aligned}
\end{eqnarray}
Therefore, using that
\begin{align*}
	\int_0^T\int_{\Omega}f(u)u_tdxdt=\int_{\Omega}F(u(T))dx-\int_{\Omega}F(u(0))dx; ~~~
	\int_0^T\int_{\Omega}hu_tdxdt=\int_{\Omega}hu(T)dx-\int_{\Omega}hu(0)dx,
\end{align*}
we obtain from \eqref{energy-equal-7} the desired energy identity for weak solutions \eqref{energy-equality}. This concludes the proof of Part ${\bf(II)}$.

\subsection{Proof of Theorem \ref{theo-global}-{\bf(III)}: Hadamard Well-posedness}\label{cont-dependence}
Hadamard well-posedness of weak solutions includes two additional properties of the solutions: uniqueness and continuous dependence on the data [often referred to as stability or robustness]. In fact, the key property is uniqueness. Having the energy identity and the uniqueness of weak solutions allows one to deduce robustness \cite{Koch-lasiecka}. The latter needs to be shown for the already constructed solutions--and not necessarily from the definition of weak solutions [due to uniqueness]. Let $u^1(t)\neq u^2(t)$ be two distinct weak solutions of problem \eqref{P}, with
\[
(u^1(0),u_t^1(0))=(u_0^1,u_1^1)\in\mathcal H,
\qquad
(u^2(0),u_t^2(0))=(u_0^2,u_1^2)\in\mathcal H.
\]
Then the difference $z(t):=u^1(t)-u^2(t)$ satisfies
\begin{equation}
	\label{approximate-diff}
	\begin{cases}
		z_{tt}-\Delta z+\bigl[g(u_t^1)-g(u_t^2)\bigr]+\bigl[f(u^1)-f(u^2)\bigr]=0,
		& \text{in } \Omega\times(0,\infty),\\[1mm]
		z=0,
		& \text{on } \Gamma\times(0,\infty),\\[1mm]
		z(\cdot,0)=u_0^1-u_0^2,\qquad
		z_t(\cdot,0)=u_1^1-u_1^2,
		& \text{in } \Omega.
	\end{cases}
\end{equation}
Using the same arguments as in the previous section for the proof of the energy identity \eqref{energy-equality}, now applied to the difference $z(t)$, we deduce from \eqref{approximate-diff} the following identity, valid for all weak solutions:
\begin{equation}
	\label{est2-1}
	\begin{aligned}
		&\frac12\|U^1(T)-U^2(T)\|_{\mathcal H}^{2}
		+\int_{0}^{T}\int_{\Omega}\bigl[g(u_t^1)-g(u_t^2)\bigr]z_t\,dx\,ds \\
		&\qquad
		+\int_{0}^{T}\int_{\Omega}\bigl[f(u^1)-f(u^2)\bigr]z_t\,dx\,ds
		=\frac12\|U^1(0)-U^2(0)\|_{\mathcal H}^{2}.
	\end{aligned}
\end{equation}
Proceeding term by term in the identity \eqref{est2-1}, exactly as in the derivation of the estimates from \eqref{conv-g-d-1} to \eqref{2gn5-1}, with $U^1$ and $U^2$ in place of $U^m$ and $U^n$, respectively, we obtain
\begin{align*}
	\|U^1(t)-U^2(t)\|_{\mathcal H}^{2}
	\le C_1e^{C_2t}\|U^1(0)-U^2(0)\|_{\mathcal H}^{2},
	\qquad t\in[0,T],
\end{align*}
where $C_j=C_j(\|U_0^1\|_{\mathcal H},\|U_0^2\|_{\mathcal H})$, $j=1,2$. This establishes \eqref{Lipschitz-property} and completes the proof of Part {\bf (III)}. Therefore, the proof of Theorem \ref{theo-global} is complete.
\qed

\begin{remark}
	\begin{enumerate}
		\item
		The arguments given above provide uniqueness and continuous dependence in \enquote{one shot}. This should be contrasted with \cite{Lorena-lasiecka} where continuous dependence requires an  additional regularity of solutions.
		\item
		The proof of the theorem shows that every weak solution arises as the strong limit of regular Galerkin approximations. This fact is particularly useful in the construction of smooth approximating sequences for weak solutions.
		\item
		The validity of the energy identity allows one to upgrade weak continuity in time to strong continuity in the natural energy space, thereby removing the symbol $C_w$ from the characterization of weak  solutions. Consequently, the associated solution operator generates a well-defined dynamical system.
	\end{enumerate}
\end{remark}

\subsection{Generation of a Dynamical System} 
Theorem \ref{theo-global} guarantees that weak solutions to problem \eqref{P} generate a dynamical system $(\mathcal{H},S_t)$. The evolution operator $S_t : \mathcal{H} \to \mathcal{H}$ is characterized by the relation
\begin{align*}
	S_t U_0 = (u(t), u_t(t))=U(t), \quad U_0 \in \mathcal{H},
\end{align*}
where $U$ denotes the unique weak solution to equation \eqref{P}. Moreover, it directly follows from Theorem \ref{theo-global}-${\bf(III)}$ that the dynamical system $(\mathcal{H}, S_t)$ exhibits the  local \textbf{Lipschitz continuity property} on every interval $[0,T]$.
\begin{equation}\label{Lipschitz-inequality}
	\| S_t U_0^1 - S_t U_0^2 \|_{\mathcal{H}}^2 \leq e^{c_Rt} \| U_0^1 - U_0^2 \|_{\mathcal{H}}^2, \quad \forall t \in [0, T],
\end{equation}
where $U_0^i \in \mathcal{H}$ with $\|U^{i}_0 \|_{\mathcal{H}} \leq R$, $i=1,2$ and $c_{R}$ is a positive constant depending on $R$.

\section{Proof of Theorem \ref{theo_main1}: Global Attractors} 
The existence of a compact global attractor is established by the results presented in Propositions \ref{Prop-asymp-compact}, \ref{Prop-gradient-system}, \ref{Prop0}, and \ref{Prop-stationary solutions}.
Proposition \ref{Prop-asymp-compact} shows that the dynamical system $(\mathcal{H},S_t)$ is asymptotically compact, while Proposition \ref{Prop-gradient-system} ensures that
$(\mathcal{H},S_t)$ is gradient. From Proposition \ref{Prop0}, $\Phi(U)=E(U)$ is a Lyapunov function bounded from above on any bounded set of $\mathcal{H}$ and the set $\Phi_R=\{U:\Phi(U)\le R\}$ is bounded for every $R$. Finally, it follows from Proposition \ref{Prop-stationary solutions} that the set
$\mathcal{N}$ is bounded and there exists an absorbing ball for the system  Corollary 2.29 \cite{chueshov-white}. Therefore, it follows directly from \cite[Corollary 7.5.7]{chueshov-yellow} that the system $(\mathcal{H},S_t)$ has a compact global attractor
$$\mathfrak{A}=\mathrm{M}^u(\mathcal{N}).$$
In particular, for the case $C_{\nu}=0$ and $h\equiv 0$ the attractor is the trivial $\mathfrak{A}=\{(0,0)\}.$
\begin{remark}
	In Appendix~\ref{appendix}, we present a direct construction of an absorbing ball, which also provides specific parameters delineating the construct. 
\end{remark}


\subsection{Asymptotic Compactness}
\begin{proposition} \label{Prop-asymp-compact} Let the assumptions of Theorem \ref{theo-global} hold. Then, the dynamical system $(\mathcal{H},S_t)$ associated with problem \eqref{P} is asymptotically compact. 
\end{proposition}
\begin{proof}	
	From Theorem \ref{theo-global}-${\bf(IV)}$ every weak solution $U(t) =(u(t),u_t(t))$ of Eq. \eqref{P} satisfies the energy {\bf equality} \eqref{energy-equality}.
	Consider the standard  perturbed Lyapunov's function $V$ with a perturbation $\epsilon$ given by
	$$ V(U) := E(U)  +\epsilon \int_{\Omega}u_tudx.$$
	Since identity \eqref{energy-equality} is valid for weak solutions of problem \eqref{P} and perturbation is given by compact term, 
	standard argument [using the equation] leads to the energy {\it identity} satisfied by each weak solution in the the perturbed functional. 
	\begin{equation}\label{V}
		V(U(T)) =e^{-\epsilon T} V(U(0) )   +\int_0^T e^{-\epsilon(T-t)} H(U(t))dt,
	\end{equation}
	where $H(U)=L(U)+M(U)$ with
	\begin{eqnarray*}
		L(U)&:=& -\int_{\Omega}g(u_t)u_tdx+\frac{3\epsilon}{2} ||u_t||^2,\\
		M (U) &:=& -\epsilon \left[ \int_{\Omega}f(u)udx-\int_{\Omega}F(u)dx \right]   -\frac{\epsilon}{2}\|\nabla u\|^2-\epsilon \int_{\Omega}g(u_t)udx+\epsilon^2 \int_{\omega}u_tudx.
	\end{eqnarray*}
	Our goal is to show that a semiflow  $\mathcal{G}$- family of maps $\varphi:[0,\infty) \rightarrow \cH $ [weak solutions] is asymptotically compact. We shall follow closely  [including the notation] the method presented in \cite{ball} applied to linear dissipation. We take a sequence $\varphi_j \in \mathcal{G}$ with $\varphi_j(0)$ bounded, and let $t_j \rightarrow  \infty$. By (\ref{energy-equality})
	$E(\varphi_j(t_j ))$ is bounded which also implies $\varphi_j(t_j) $ is bounded in $\cH$ .
	This allows to select weakly convergent subsequence
	\begin{equation}\label{lambda}
		\varphi_j (t_j) \rightharpoonup \lambda, \quad\lambda \in \cH
	\end{equation}
	and by continuity of the semiflow, there exists $\lambda_T \in \cH $ such that 
	\begin{align*}
		\varphi_j (t_j-T ) \rightharpoonup \lambda _T, \quad \lambda_T \in \cH.
	\end{align*}
	Since the trajectories   generated by $\varphi_j \in\mathcal{G} $,   with  weakly convergent initial data,  are also  weakly convergent [on a subsequence] to a trajectory in $\mathcal{G} $  [  property  \cite[(C4w)]{ball} ]  [ The latter is justified on the virtue of the estimates on weak solutions in section 3 and fundamental theorem of calculus], we can still assume that there exists a trajectory $\widetilde{\varphi}(t)=(u(t),u_t(t))\in \mathcal{G}$ with
	$$ U^j(t) \equiv \varphi_j(t_j + t-T )  \rightharpoonup \widetilde{\varphi}(t)=(u(t),u_t(t))=U(t) \in \mathcal{G},$$
	where $\widetilde{\varphi}(0) = \lambda_T$ and $\widetilde{\varphi}(T) =\lambda.$
	Applying the perturbed energy equality \eqref{V} to the above solutions by substituting $U^j(t)\equiv \varphi_j(t_j + t -T)$ we obtain
	\begin{align*}
		V(\varphi_j(t_j))=e^{-\epsilon T} V(\varphi_j(t_j-T)+\int_0^T e^{-\epsilon(T-t)} H(U^j(t)) dt
	\end{align*}
	and the  {\it equality} for limiting solutions [using again energy equality  in (\ref{V})] 
	\begin{align*}
		V (\lambda)= V (\widetilde{\varphi}(T))  =e^{-\epsilon T} V(\lambda_T )   +\int_0^T e^{-\epsilon(T-t)} H(U(t) )
		dt.
	\end{align*}
	
	Our goal is to show that
	\begin{equation}\label{goal}
		\lim \sup_{j\rightarrow \infty}  V(\varphi_j(t_j) ) \leq  V(\lambda)  \leq \lim \inf_{j\rightarrow \infty } V (\varphi_j(t_j)).
	\end{equation}
	
	The second inequality follows from weak lower semicontinuity of $V(\varphi)$ and weak convergence in (\ref{lambda}). Also, in the case when $H$ is compact, the inequality in (\ref{goal}) follows quickly by taking $T$ to infinity. In non-compact case, more work is needed as noted in \cite{ball}. The key observation is the control of \enquote{sign} which allows to use version of Fatou's lemma.
	We shall use the elementary relation:
	$\lim \sup ( F) = - \lim \inf (-F).$
	The main task is to pass with the  weak limit under the integral sign. For transparency, we shall  focus on critical-noncompact terms, where passage with the weak limit requires additional [to weak convergence] properties.
	In short, the last two terms in $M(U) $ owning to compact embedding $H^1(\Omega) \subset L^5(\Omega)$ satisfy
	$$\lim_j \left[\epsilon \int_{\Omega}g(u_t^j)u^jdx +\epsilon^2 \int_{\Omega}u_t^{j}u^jdx \right]  =\epsilon \int_{\Omega}g(u_t)udx +\epsilon^2 \int_{\Omega}u_tudx.$$
	
	Since   $\displaystyle{-\epsilon \left[ \int_{\Omega}f(u^j)u^jdx-\int_{\Omega}F(u^j)dx\right]-\frac{\epsilon}{2} \|\nabla u^j\|^2}$ is negative, weak lower semicontinuity implies
	\begin{eqnarray*}
		&&\int_0^T e^{-\epsilon(T-t)} \left[ \epsilon  \int_{\Omega}(f(u)u-F(u))dx  +\frac{\epsilon}{2} \|\nabla u(t)\|^2 \right]dt\\
		&&\leq \lim \inf _j  \int_0^T e^{-\epsilon (T-t) } \left[ \epsilon  \int_{\Omega}(f(u^j)u^j-F(u^j))dx  +\frac{\epsilon}{2} \|\nabla u^j(t)\|^2 \right]dt\\
		&& =- \lim \sup_{j}  \int_0^T e^{-\epsilon (T-t) }\left[ -\epsilon  \int_{\Omega}(f(u^j)u^j-F(u^j))dx-\frac{\epsilon}{2} \|\nabla u^j(t)\|^2 \right]dt.
	\end{eqnarray*}
	This gives
	\begin{eqnarray*}-\int_0^T e^{-\epsilon (T-t) }M(U)dt &\leq& \lim \inf_{j}  \int_0^T e^{-\epsilon (T-t)} (-M(U^j) )dt
		= -  \lim \sup_{j} \int_0^T e^{-\epsilon (T-t) }M (U^j)dt.
	\end{eqnarray*}
	Hence
	\begin{equation}\label{M}
		\lim \sup_{j}\int_0^T e^{-\epsilon (T-t) }M(U^j)dt \leq  \int_0^T e^{-\epsilon (T-t)}  M(U) dt.
	\end{equation}
	We shall next estimate the damping terms: For any small constant $\eta > 0$
	$$-L(U) =\int_{\Omega}g(u_t)u_tdx-\frac{3}{2}\epsilon||u_t(t)||^2  \geq  \int_{\Omega}[ \,|u_t|^6 - \eta  |u_t|^6  -\epsilon  C_{\eta } \,]dx.$$
	Taking $\eta < 1 $ gives $\displaystyle{\left[\,\int_{\Omega}g(u_t^j)u_t^jdx - \frac{3\epsilon}{2}||u_t^j||^2 + \epsilon C_{\eta } \,\right] \geq 0}$ and via  weak lowersemicontinuity
	
	$$-\int_0^T e^{-\epsilon (T-t) }L(U)dt +\epsilon C_{\eta }   \leq \lim \inf_{j} \int_0^T e^{-\epsilon (T-t) }[-L(U^j)  +\epsilon C_{\eta}]dt$$
	and converting into lim sup
	\begin{equation}\label{L}\int_0^T e^{-\epsilon (T-t) }L(U)dt \geq \lim \sup_{j} \int_0^T e^{-\epsilon (T-t)}L(U^j)dt.\end{equation}
	Combining \eqref{M} and \eqref{L}, we get
	$$\lim \sup_j \int_0^T e^{-\epsilon (T-t) }\left[\,L(U^j)+M(U^j)\,\right]dt \leq  \int_0^T e^{-\epsilon (T-t) }\left[\,L(U)+M(U)\,\right]dt.$$
	This yields
	
	$$ \lim \sup_j \int_0^T e^{-\epsilon (T-t)}H (U^j)dt \leq
	\int_0^Te^{-\epsilon (T-t)}H(U)dt.$$
	
	By taking $T\rightarrow \infty$ and using boundedness of trajectories in $\mathcal{G}$ one obtains the conclusion in (\ref{goal}). The final result follows from the fact that weak convergence and norm convergence imply strong convergence.
\end{proof}

\subsection{Gradient System}
\begin{proposition}\label{Prop-gradient-system} Assume that the assumptions of Theorem \ref{theo-global} hold. Then, 
	$(\mathcal{H},S_t)$ is a gradient dynamical system.
\end{proposition}
\begin{proof} Let us take $\Phi$ as the energy functional $E$ defined in \eqref{functional-energy}. From energy equality \eqref{energy-equality}, we have
	\begin{align}\label{inequality grad2} 	\Phi(S_tU_0)
		+\int_0^t\int_{\Omega}g(u_t)u_tdxd\tau =  \Phi(U_0),
	\end{align}
	for every $U_0\in \mathcal{H}$. Hence $\Phi(S_tU_0)$ is non increasing. Now let us suppose $\Phi(S_tU_0)=\Phi(U_0)$ for all $t\ge 0$. Then from above
	From Assumption \eqref{hyp_g'}, we have
	\begin{align}\label{est-grad-2}
		0\le \kappa_0\int_0^t\int_{\Omega}|u_t|^{6}dxd\tau\le\int_0^t\int_{\Omega}g(u_t)u_tdxd\tau=0.
	\end{align}
	Then, using embedding $L^6(\Omega)\hookrightarrow H^0$, it follows from \eqref{est-grad-2} that
	$$
	\|u_t(t)\|=0\ \ \hbox{for almost everywhere}\ \ t>0.
	$$
	Moreover, it follows from $u_t\in C([0, T];H^0)$ for all $T>0$ that $\|u_t(t)\|=0$ for all $t>0$, which implies that $U_0\in \mathcal{N}$, where $\mathcal{N}$ is the set of stationary points of the dynamical system $(\mathcal{H},S_t)$. Using that
	$$
	\quad U_0\in \mathcal{N} \quad  \Leftrightarrow \quad S_t(U_0)=U_0,   \quad  t>0,
	$$
	then   $\Phi$ is a {\it  strict Lyapunov } functional for the dynamical system $(\mathcal{H},S_t)$.
\end{proof}
\subsection{Boundedness of the set of stationary solutions}
\begin{proposition}\label{Prop0}
	Let $\Phi$ be the Lyapunov function given in Proposition \ref{Prop-gradient-system}. Then $\Phi(U)$ is bounded from above on any bounded subset of $\mathcal{H}$ and the set $\Phi_R=\{U:\Phi(U)\le R\}$ is bounded for every $R$.
\end{proposition}	
\begin{proof}
	Let $R>0$ and we define $\Phi_R=\{U\in \mathcal{H}|\Phi(U)\le R\}.$
	From Assumption \eqref{hyp_f2} and H\"older’s inequality, it follows from \eqref{inequality grad2} that
	\begin{align*}
		\frac{\omega}{4}||U||^2_{\mathcal{H}}\le \Phi(U)+C_{\nu}|\Omega|+\frac{1}{\omega\lambda_1}\|h\|^2\le R+C_{\nu}|\Omega|+\frac{1}{\omega\lambda_1}\|h\|^2.
	\end{align*}
	Therefore, for $U(t)\in \Phi_R$, we get
	\begin{align*}
		||U||^2_{\mathcal{H}}\le \frac{4}{\omega}\left[R+C_{\nu}|\Omega|+\frac{1}{\omega\lambda_1}\|h\|^2\right].
	\end{align*}
	Which implies that $\Phi_R$ is bounded for every $R>0$.
\end{proof}
\begin{proposition}
	\label{Prop-stationary solutions}
	The set $\mathcal{N}$ of stationary solutions of problem \eqref{P} is bounded in $\mathcal{H}.$
\end{proposition}
\begin{proof}
	Let $u$ be a stationary solution of \eqref{P}. Then, by the definition of $\mathcal{N}$, we have
	\begin{eqnarray}\label{est-1}
		-\Delta u+f(u)=h,\quad u\in H^1.
	\end{eqnarray}
	By elliptic regularity, we infer that $u\in H^2$. Indeed, this follows from the fact that the right-hand side $h - f(u)$ belongs to $H^0$, which is ensured by the dissipative condition \eqref{hyp-inf-f}.
	Multiplying \eqref{est-1} by $u$ and integrating over $\Omega$, we get
	\begin{eqnarray}\label{est-22}
		\|\nabla u\|^2 +\int_{\Omega}f(u)udx=\int_{\Omega}hudx.
	\end{eqnarray}
	From Assumption \eqref{hyp_f2} and H\"older inequality, it follows from \eqref{est-22} that 
	\begin{eqnarray}\label{est-5}
		\|\nabla u\|^2\le \frac{1}{\omega^2\lambda_1}\|h\|^2+\frac{2C_{\nu}|\Omega|}{\omega}.
	\end{eqnarray}
\end{proof}


\section{Proof of Theorem \ref{theo-quasi}: Quasi-stability property}
The existence of a compact global attractor $\mathfrak{A}$ is guaranteed by Theorem \ref{theo_main1} under weaker assumption $g'(0) \geq 0 $. However, a strict inequality 
$g'(0) >0 $ allows to show that the dynamical system $(\mathcal{H},S_t)$ is quasi-stable as defined in \cite[Definition 7.9.2]{chueshov-yellow}. And within the context of the theory of dynamical systems associated with the study of attractors, for quasi-stable systems, it is possible to obtain important properties such as the finite dimension and smoothness of the attractor, as well as the existence of exponential attractors. Properties that will be established for the dynamical system associated with problem \eqref{P} in Theorem \ref{theo_main2}.

The proof that the system $(\mathcal{H},S_t)$ is quasi-stable follows from the fact that $(\mathcal{H},S_t)$ satisfies conditions (7.9.2) and (7.9.3) from \cite[Definition 7.9.2]{chueshov-yellow}.  The condition (7.9.2) is guaranteed directly from the Lipschitz property \eqref{Lipschitz-inequality}. It remains to prove the quasi-stability inequality (7.9.3) which is ensured by the relation \eqref{inequality-main} established in the following proposition.
\subsection{Quasi-stability Estimate}
\begin{proposition}\label{stabilizability-estimate} Let $(u^1(t),u^1_t(t))=S_t(u^1_0,u^1_1)$ and $ (u^2(t), u^2_t(t))= S_t(u^2_0,u^2_t)$ be two weak solutions to \eqref{P} with initial data $U^1_0=(u^1_0,u^1_1), U^2_0=(u^2_0,u^2_1)$ lying in a bounded set $B\subset \mathcal{H}$. Let the hypotheses of Theorem \ref{theo-global} be valid. In addition we assume that $g'(0)>0$.  Let $z(t)=u^1(t)-u^2(t).$ Then we have the following relation
	\begin{eqnarray}\label{inequality-main}
		||S_tU^1_0-S_tU^2_0||^2_{\mathcal{H}}\le b(t)||U^1_0-U^2_0||^2_{\mathcal{H}}+c(t)\sup_{s\in [0,t]}\|u^1(s)-u^2(s)\|^2,
	\end{eqnarray}
	where $b(t)$ and $c (t)$ are nonnegative scalar functions satisfying  $b\in L^1(\mathbb{R}^+)$ with  $\displaystyle\lim_{t\rightarrow +\infty}b(t)=0$ and $c(t)$ is
	bounded on $[0,\infty]$.
\end{proposition}
\begin{proof}
	We begin with the following simple observation:
	Using that $g\in C^1(\mathbb{R})$ and $g'(0)>0$, then there exists a constant $\eta>0$ such that
	$g'(s)>0$ for all $|s|<2\eta$. And since $g'>0$ in $\mathbb{R}$, there exists a constant $m>0$ such that
	$g'(s)\ge m> 0$ for all $|s|\le \eta$. On the other hand, from Assumption \eqref{hyp_g'}, we have
	$g'(s)\ge \kappa_0|s|^{4}\ge \kappa_0\eta^4>0$ for all $|s|\ge \eta$.
	Hence, choosing $\kappa_2=\inf\{m,\kappa_0\eta^4\}$, we infer
	\begin{align}\label{hyp-g'-2}g'(s)\ge \kappa_2,\quad \mbox{for all}\quad s\in \mathbb{R}.\end{align}
	Therefore, from \eqref{hyp-g'-2} and using the same arguments used in \cite[Remark 4:1]{EECT-2017} with $\gamma=4$, we obtain
	\begin{align}
		\label{hyp-g'-22}
		\left[\,g(r)-g(s)\,\right](r-s)\ge \kappa_2|s-r|^2+\frac{\kappa_0}{10}\left(|r|^4+|s|^4\right)|r-s|^2.
	\end{align}

	Let $z=u^1-u^2.$ Then $z$ satisfies the equation
	\begin{eqnarray}\left\{\begin{array}{l}\label{diff-weak-solution}\displaystyle{
				z_{tt}-\Delta z+\left[\,g(u^1_t)-g(u^2_t)\,\right]+\left[\,f(u^1)-f(u^2)\,\right]=0,\quad \mbox{in}\quad \Omega\times [0,\infty),}\\
			\displaystyle{ (z(0),z_t(0))=(u_{0}-v_{0},u_{1}-v_{1}),\quad z=0,\quad \mbox{on}\quad \Gamma.}\end{array}\right.
	\end{eqnarray}
	Let $\chi_{\theta}(t):=\theta u^1(t)+(1-\theta)u^2(t)$. Using the energy identity \eqref{energy-equality} for the difference of two weak solutions and integrating the relation obtained from $\tau\ge 0$ to $T$, we get 
		\begin{align}
			\label{ASP-11}
			E_z(T)-E_z(\tau)+\int_{\tau}^T\int_{\Omega}D(z_t(s))dxds=\;\underbrace{K_f\int_{\tau}^T\int_{\Omega}zz_tds}_{J_1}
			+\underbrace{\frac{1}{2}\int_{\tau}^T\int_{\Omega}\int_0^1f''(\chi_{\theta})(\chi_{\theta})_td\theta|z|^2dxds}_{J_2},
		\end{align}
		where $D(z_t):=\left[g(u^1_t)-g(u^2_t)\right]z_t$
		and
		$$E_z:=\frac{1}{2}||S_tU^1_0-S_tU^2_0||^2_{\mathcal{H}}+\frac{1}{2}\int_{\Omega}\int_0^1f'(\chi_{\theta})d\theta|z|^2dx+\frac{K_f}{2}\|z\|^2,$$
		with the constant $K_f$  obtained from dissipativity condition \eqref{hyp-inf-f} as in \eqref{strong-4}.
		
		Using H\"older inequality, \eqref{hyp-g'-22}, and Young's inequality, we have
		\begin{eqnarray*}
			|J_1|
			\le \frac{K_f^2}{2\kappa_2}\int_{\tau}^T\|z\|^2ds+\frac{1}{2}\int_{\tau}^T\int_{\Omega}D(z_t(s))dxds.
		\end{eqnarray*}
		From Assumption \eqref{hyp_f''}, H\"older inequality with $\frac{1}{2}+\frac{1}{6}+\frac{1}{3}=1$, embedding $H^1\hookrightarrow L^6(\Omega)$, and Young's inequality, we have
		\begin{eqnarray*}
			|J_2|
			&\le& \delta\int_{\tau}^TE_z(s)ds+ C_{B,\delta}\int_{\tau}^Td(s,u^1,u^2)E_z(s)ds,
		\end{eqnarray*}
		where
		$d(t,u^1,u^2):=\|u^1_t(t)\|^6_6+\|u^2_t(t)\|^6_6.$
		Then, substituting $|J_1|$ and $|J_2|$ into \eqref{ASP-11}, we find
		\begin{eqnarray}
			\begin{aligned}
				\label{ASP-111}
				E_z(T)-E_z(\tau)+\frac{1}{2}\int_{\tau}^T\int_{\Omega}D(z_t(s))dxds
				\le &\;\frac{K_f^2}{2\kappa_2}\int_{\tau}^T\|z\|^2ds+\delta\int_{\tau}^TE_z(s)ds\\
				&+ C_{B,\delta}\int_{\tau}^Td(s,u^1,u^2)E_z(s)ds.
			\end{aligned}
		\end{eqnarray}
		Now, integrating \eqref{ASP-111} from $0$ to $T$, we obtain
		\begin{eqnarray}
			\begin{aligned}
				\label{ASP-1111}
				&TE_z(T)+\frac{1}{2}\int_0^T\int_{\tau}^T\int_{\Omega}D(z_t(s))dxdsd\tau\;\le\;\int_0^TE(\tau)d\tau \\
				&+\frac{K_f^2}{2\kappa_2}\int_0^T\int_{\tau}^T\|z\|^2dsd\tau+\delta T\int_{0}^TE_z(s)ds+ TC_{B,\delta}\int_0^Td(s,u^1,u^2)E_z(s)ds.
			\end{aligned}
		\end{eqnarray}
		Next, multiplying Eq. \eqref{diff-weak-solution} by $z$ and integrating over $\Omega\times [0,T]$, we have
		\begin{eqnarray}
			\label{ASP-2}
			\begin{aligned}
				2\int_0^TE_z(s)ds=&\;2\int_0^T\|z_t\|^2ds+K_f\int_0^T\|z\|^2ds\\
				&-\underbrace{\left[\int_{\Omega}z_t(s)z(s)dx\right]^{s=T}_{s=0}}_{J_3}-\underbrace{\int_0^T\int_{\Omega}\left[g(u^1_t)-g(u^2_t)\right]zdxds}_{J_4}.
			\end{aligned}
		\end{eqnarray}
		From \eqref{hyp-g'-22}, we get
		\begin{eqnarray}\label{j0}
			2\int_0^T\|z_t\|^2ds\le \frac{2}{\kappa_2} \int_0^T\int_{\Omega}D(z_t(s))dxds.
		\end{eqnarray}
		From the continuous embedding $H^1\hookrightarrow H^0$ and the equivalence $E_z\sim \|(z,z_t)\|_{\mathcal H}^2$, it follows that
		\begin{align}
			\label{Asymptotic-2}
			|J_3|\le \frac{1}{\lambda_1^{1/2}}\left[E_z(T)+E_z(0)\right].
		\end{align}
		Moreover, using \eqref{ASP-11}, together with the estimates for $J_1$ and $J_2$ (with $\tau=0$), and choosing $\delta=\frac{\lambda_1^{1/2}}{2}$, we infer that
		\begin{eqnarray}
			\label{Asymptotic-22}
			\begin{aligned}
				E_z(0)\le&\; E_z(T)+\frac{3}{2}\int_0^T\int_{\Omega}D(z_t(s))\,dx\,ds
				+\frac{K_f^2}{2\kappa_2}\int_0^T\|z\|^2\,ds\\
				&+\frac{\lambda_1^{1/2}}{2}\int_0^TE_z(s)\,ds
				+C_B\int_0^Td(s,u^1,u^2)E_z(s)\,ds.
			\end{aligned}
		\end{eqnarray}
		Substituting \eqref{Asymptotic-22} into \eqref{Asymptotic-2}, we obtain
		\begin{eqnarray}
			\label{j3}
			\begin{aligned}
				|J_3|\le&\; \frac{2}{\lambda_1^{1/2}}E_z(T)
				+\frac{3}{2\lambda_1^{1/2}}\int_0^T\int_{\Omega}D(z_t(s))\,dx\,ds
				+\frac{K_f^2}{2\kappa_2\lambda_1^{1/2}}\int_0^T\|z\|^2\,ds\\
				&+\frac{1}{2}\int_0^TE_z(s)\,ds
				+C_B\int_0^Td(s,u^1,u^2)E_z(s)\,ds.
			\end{aligned}
		\end{eqnarray}
		From Assumption \eqref{hyp_g'} and \eqref{hyp-g'-22}, it follows that
		\begin{align}\label{j4}
			|J_4|
			\le \frac{1}{\kappa_2}\int_0^T\int_{\Omega}D(z_t)\,dx\,ds
			+\frac{\kappa_1^2}{4}\int_0^T\|z\|^2\,ds
			+\underbrace{8\kappa_1\int_0^T\int_{\Omega}\left(|u_t^1|^4+|u_t^2|^4\right)|z_t||z|\,dx\,ds}_{J_4'}.
		\end{align}
		Moreover, by the inequalities of H\"older with $\frac12+\frac12=1$ and $\frac23+\frac13=1$, the Sobolev embedding $H^1\hookrightarrow L^6(\Omega)$, the equivalence $E_z\sim \|(z,z_t)\|_{\mathcal H}^2$
		and \eqref{hyp-g'-22}, we obtain
		\begin{align}
			\label{j4'}
			J_4'
			\le \frac{40\kappa_1}{\kappa_0}\int_0^T\int_{\Omega}D(z_t(s))\,dx\,ds
			+\frac{32\kappa_1^2}{\lambda^2}\int_0^Td(s,u^1,u^2)E_z(s)\,ds
			+\frac{1}{2}\int_0^TE_z(s)\,ds.
		\end{align}
		Substituting \eqref{j0}, \eqref{j3}, \eqref{j4}, and \eqref{j4'} into \eqref{ASP-2}, we arrive at
		\begin{eqnarray}\label{ASP-22}
			\begin{aligned}
				\int_0^TE_z(s)\,ds
				\le&\; \frac{2}{\lambda_1^{1/2}}E_z(T)
				+C_0\int_0^T\int_{\Omega}D(z_t(s))\,dx\,ds\\
				&+C_1\int_0^T\|z\|^2\,ds
				+C_B'\int_0^Td(s,u^1,u^2)E_z(s)\,ds,
			\end{aligned}
		\end{eqnarray}
		where
		$$
		C_0=\frac{2}{\kappa_2}+\frac{3}{2\lambda_1^{1/2}}+\frac{1}{\kappa_2}+\frac{40\kappa_1}{\kappa_0},\quad
		C_1=K_f+\frac{K_f^2}{2\kappa_2\lambda_1^{1/2}}+\frac{\kappa_1^2}{4},\quad
		\mbox{
			and}
		\quad
		C_B'=C_B+\frac{32\kappa_1^2}{\lambda^2}.
		$$
		Combining \eqref{ASP-1111} and \eqref{ASP-22} we obtain that
		\begin{eqnarray}
			\begin{aligned}
				\label{ASP-33}
				&\left(T-\frac{4}{\lambda_1^{1/2}}\right)E_z(t)+\left(1-\delta T\right)\int_0^TE(s)ds\le \;2C_0 \int_0^T\int_{\Omega}D(z_t(s))dxds\\
				&\quad+\left(TC_{B,\delta}+2C'_B\right)\int_0^Td(s,u^1,u^2)E_z(s)ds
				+\left(2C_1+\frac{K_fT}{2\kappa_2}\right)\int_0^T\|z\|^2ds.
			\end{aligned}
		\end{eqnarray}
		Using without loss of generality that $\frac{T}{2}-\frac{4}{\lambda_1^{1/2}}>0$ and taking $\delta$ small enough such that $1-\delta T\ge \frac{1}{2}$, it follows from \eqref{ASP-33} that
		\begin{eqnarray}
			\begin{aligned}
				\label{ASP-333}
				&TE_z(T)+\int_0^TE(s)ds\le \;4C_0 \int_0^T\int_{\Omega}D(z_t(s))dxds\\
				&\quad+2\left(TC_{B,\delta}+2C'_B\right)\int_0^Td(s,u^1,u^2)E_z(s)ds
				+2\left(2C_1+\frac{K_fT}{2\kappa_2}\right)\int_0^T\|z\|^2ds.
			\end{aligned}
		\end{eqnarray}
		From \eqref{ASP-111} with $\tau=0$, we have
		\begin{align}
			\label{ASP-11111}
			\frac{1}{2}\int_{0}^T\int_{\Omega}D(z_t(s))dxds
			\le&\; E_z(0)-E_z(T)+\delta\int_{0}^TE_z(s)ds
			+ C_{B,\delta}\int_{0}^Td(s,u^1,u^2)E_z(s)ds.
		\end{align}
		Substituting \eqref{ASP-11111} into \eqref{ASP-333} and choosing $\delta=\frac{1}{2C_0}$, we obtain
		\begin{eqnarray*}
			E_z(T)\le\frac{2C_0}{T+2C_0} E_z(0)+C_B\left[\sup_{s\in[0,T]}\|z(s)\|^2+\int_0^Td(s,u^1,u^2)E_z(s)ds\right],
		\end{eqnarray*}
		where $C_B=\frac{4C_0\left(TC_{B,\delta}+2C'_B\right)}{T+2C_0}+\frac{4C_0\left(2C_1+\frac{K_fT}{2\kappa_2}\right)}{T+2C_0}$. 
		
		Reiterating the estimate on the intervals $[mT,(m+1)T]$ yields
		\begin{align*}
			E_z((m+1)T)\le \gamma E_z(mT)+C_{B}b_{m},\quad m=0,1,2,\cdots,
		\end{align*}
		where 
		$$0<\gamma\equiv\frac{2C_{0}}{T+2C_{0}}<1\quad\mbox{and}\quad b_m\equiv \sup_{s\in [mT,(m+1)T]}\|z(s)\|^2+\int_{mT}^{(m+1)T}d(s;u^1,u^2)E_z(s)ds.$$
		This yields
		$$E_z(mT)\le \gamma^{m}E_z(0)+c\sum_{l=1}^m\eta^{m-l}b_{l-1}.$$
		Since $\gamma<1$, using the same argument as in (\cite{chueshov-white}, Remark 3.30) along with the definition of $b_l$ we obtain that there exists $\varrho>0$ such that
		\begin{align*}
			E_z(t)\le C_1e^{-\varrho t}E_z(0)+C_2\left[\,\sup_{0\le s\le t}\|z(s)\|^2+\int_0^te^{-\varrho(t-s)}d(s,u^1,u^2)E_z(s)ds\,\right],
		\end{align*}
		for all $t\ge 0$. Therefore, applying Gronwall's lemma we find
		\begin{align*}
			E_z(t)\le \left[\,C_1e^{-\varrho t}E_z(0)+C_2\sup_{0\le s\le t}\|z(s)\|^2\,\right]e^{C_2\int_0^td(s,u^1,u^2)ds}.
		\end{align*}
		Using that $\frac{1}{2}||U^1-U^2||^2_{\mathcal{H}}\le E_z(t)\le C_B||U^1-U^2||^2_{\mathcal{H}}$, we have
		\begin{align*}
			||U^1-U^2||^2_{\mathcal{H}}\le b(t)||U^1_0-U^2_0||^2_{\mathcal{H}}+c(t)\sup_{0\le s\le t}\|z(s)\|^2.
		\end{align*}
		where
		\begin{align}b(t):=2C_BC_1e^{-\varrho t}e^{C_2\int_0^td(s,u^1,u^2)ds}\quad \mbox{and}\quad c(t):=2C_2e^{C_2\int_0^td(s,u^1,u^2)ds}.\label{definiton-c(t)}\end{align}
		Thus, using that  $d(s,u^1,u^2)=\|u^1_t\|^6_6+\|u^2_t\|^6_6\in L^1(0,t)$,  {\bf uniformly in $t>0$} we obtain
		$b(t)\in L^1(\mathbb{R}^+)$, $\lim_{t\rightarrow +\infty}b(t)=0$,
		and $c(t)$ is bounded on $\mathbb{R}^+$ due to $L^1(\mathbb{R}^+)$ integrability  of $d(s,u^1,u^2) $.
		The proof of Proposition \ref{stabilizability-estimate} is  now complete.
	\end{proof}
	
	\section{Proof of Theorem \ref{theo_main2}: Properties of Attractors}
	\subsection {Proof of Part (I) and Part (III) in Theorem \ref{theo_main2}.}
	On the virtue of Theorem \ref{theo-quasi}, we have established that under  the condition $g'(0)>0$ the system $(\mathcal{H},S_t)$ is  \enquote{quasi-stable}. Therefore, the finite dimension of the attractor $\mathfrak{A}$ (Part {\bf (I)}) is guaranteed by direct application of \cite[Theorem 7.9.6]{chueshov-yellow}. Indeed, quasistability inequality obtained for elements  residing on  a compact attractor implies, on the virtue of Thm 7.9.6 in \cite{chueshov-yellow}, finite fractal dimension of the said attractor.
	As to Part {\bf (II)} we are in a position to obtain at this point only partial result as stated in (\ref{regularity-trajec-attrac-2}).  To this end we shall use  Theorem 7.9.8 \cite{chueshov-yellow} which is an upgrade of Theorem 7.9.6 in the same reference.  With an additional requirement that the function $c(t)$  is bounded globally, one obtains additional regularity properties of the attractor.   It follows from Proposition \ref{Prop-absorbing-set} (Appendix \ref{appendix})  that the uniform estimate
	\begin{align}\label{regularity-trajec-attrac-1}||U(t)||^2_{\mathcal{H}}\le R^2,\quad t\in \mathbb{R},\end{align}
	is valid for any trajectory $\{U(t):t\in \mathbb{R}\}$ lying in the attractor $\mathfrak{A}$.
	In addition to the quasi-stability property of the system $(\mathcal{H},S_t)$, using that the function $c(t)$ defined in \eqref{definiton-c(t)} has the property $c_{\infty}=\sup_{t\in \mathbb{R}^+}c(t)<\infty$. Thus,  it follows by direct application of \cite[Theorem 7.9.8]{chueshov-yellow} that any complete trajectory $\{U(t):t\in \mathbb{R}\}$ that belongs to the global attractor $\mathfrak{A}$ has the following regularity properties
	$$u_t\in L^{\infty}(\mathbb{R};H^1)\cap C(\mathbb{R};H^0),\quad u_{tt}\in L^{\infty}(\mathbb{R};H^0).$$
	Moreover, there exists $R_1>0$ such that
	\begin{align}\label{regularity-trajec-attrac-2}
		||U_t(t)||^2_{\mathcal{H}}=\|u_{t}(t)\|_{H^1}+\|u_{tt}(t)\|^2\le R^2_1,\quad t\in \mathbb{R},
	\end{align}
	where $R_1$ depends on the constant $c_{\infty}$. To complete the proof of Part {\bf (II)}, it remains to prove that $u\in L^{\infty}(\mathbb{R};H^2)$ and $\|u(t)\|_{H^2}\le R_2$ for all $t\in \mathbb{R}$ for some $R_2>0$. Since the proof of this result requires technical  arguments, we carry out the proof separately in Section \ref{proof-H2-regularity}.\qed\\
	
	The proof of Part {\bf (III)} is an application of \cite[Theorem 7.9.9]{chueshov-yellow}. This Theorem leads to the conclusion of existence of exponential attractor, under an additional Holder type of inequality (7.8.18) in \cite{chueshov-yellow}. We shall show that this property is indeed satisfied. 
	First, let $\mathfrak{B}$ be the absorbing set obtained in the Proposition \ref{Prop-absorbing-set}. From Proposition \ref{stabilizability-estimate}, the system $(\mathcal{H},S_t)$ is quasi-stable on $\mathfrak{B}$. Now, let us consider the Hilbert spaces $\mathcal{H}_{-s}=H^{-s+1}\times H^{-s}$ with $0<s\le 1$, where $H^{2s}:=D[(-\Delta)^{s}],$ $s\in \mathbb{R}$, with Dirichlet boundary conditions.
	Next, we consider the mapping $t\in[0,T]\mapsto S_tz\in \mathcal{H}_{-1}=H^0\times H^{-1}$. Let $U(t) = S_tz=z(t)$ be a weak solution with initial data $z\in \mathfrak{B}$. It follows from Theorem \ref{theo-global} that $U_t$ has regularity $U_t=(u_t,u_{tt})\in L^{6/5}(0,T;\mathcal{H}_{-1})$. Then, we have
	\begin{eqnarray*}
		||S_{t_1}z-S_{t_2}z||_{\mathcal{H}_{-1}}&\le& \int_{t_1}^{t_2}\left\|\frac{d}{dt}z(s)\right\|_{\mathcal{H}_{-1}}ds\le \left(\int_{t_1}^{t_2}\|(u_t(s),u_{tt}(s))\|_{\mathcal{H}_{-1}}^{6/5}ds\right)^{5/6}|t_2-t_1|^{1/6}\\
		&\le& C_{\mathfrak{B}}|t_2-t_1|^{1/6},\quad 0\le t_1\le t_2\le T.
	\end{eqnarray*}
	This  shows that for each $z\in \mathfrak{B}$, the map $t\mapsto S_tz$ is H\"older continuous in the extended space $\mathcal{H}_{-1}$ with exponent $\gamma=1/6$. Now, we assume that $0<s<1$. Then, using that $\mathcal{H}_{0}\hookrightarrow \mathcal{H}_{-s}\hookrightarrow \mathcal{H}_{-1}$, applying interpolation theorem in each component of $\mathcal{H}_{-s}$ and using the H\"older continuity in $\mathcal{H}_{-1}$, we obtain
	\begin{eqnarray*}
		||S_{t_1}z-S_{t_2}z||_{\mathcal{H}_{-s}}&\le& C_s||S_{t_1}z-S_{t_2}z||^{1-s}_{\mathcal{H}}||S_{t_1}z-S_{t_2}z||^{s}_{\mathcal{H}_{-1}}\\
		&\le& C_{\mathfrak{B}}|t_2-t_1|^{s/6},\quad \mbox{with}\quad s\in (0,1].
	\end{eqnarray*}
	This shows that $t\mapsto S_tz$ is H\"older continuous in the extended spaces $\mathcal{H}_{-s}$ for $s\in (0,1]$. Therefore, the existence of a generalized exponential attractor in $\mathcal{H}_{-s}$ with $0<s\le 1$, follows from \cite[Theorem 7.9.9]{chueshov-yellow}.
	\qed
	
	\subsection{Maximal Regularity of the Attractor-Completion of the Proof of Part II}\label{proof-H2-regularity}
	To obtain this result we will use arguments similar to those used by Chueshov and Lasiecka in \cite[Section 4.2]{chueshov-lasiecka-2007}, see aso \cite{zelik}.
	The proof follows the following 3 steps. In the first, we prove the $H^2$ regularity of the stationary solutions of problem \eqref{P}; in the second, using the regularity of stationary solutions, we establish the smoothness of the trajectories for negative times $t\to -\infty$. Finally, in step 3, assuming the regularity obtained in step 2 for an initial data $U(T_0)\in H^2$ with $T_0<0$, we propagate this regularity to all $t\in \mathbb{R}.$
	
	\subsubsection*{Step 1: Regularity of stationary solutions}
	\begin{lemma}\label{lemma-stat-sol}
		The stationary solutions $u$ of problem \eqref{P} are bounded in $H^2$.
	\end{lemma}
	\begin{proof}
		Multiplying \eqref{est-1} by $-\Delta u$ and integrating over $\Omega$, we get
		\begin{eqnarray}\label{est-2}
			||\Delta u||^2-\int_{\Omega}f(u)\Delta udx=-\int_{\Omega}h\Delta udx.
		\end{eqnarray}
		Then, using Assumption \eqref{hyp_f'} together with Hölder's inequality and \eqref{est-5}, we conclude that
		\begin{align}\label{est-8}
			||\Delta u||^2\le \frac{1}{\omega^2}\|h\|^2+\frac{2K_f}{\omega}\|\nabla u\|^2\le \frac{1}{\omega^2}\|h\|^2+\frac{2K_f}{\omega}\left(\frac{1}{\omega^2\lambda_1}\|h\|^2+\frac{2C_{\nu}|\Omega|}{\omega}\right)=:\varrho,
		\end{align}
		where $K_f$ arises as in \eqref{strong-4}. This completes the proof of the lemma.
	\end{proof}
	
	\subsubsection*{Step 2: Smoothness on negative time scale}
	\paragraph{Decomposition of the solutions.}
	Let $\{U(t)=(u(t),u_t(t)):t\in \mathbb{R}\}$ be a trajectory from the global attractor $\mathfrak{A}=\mathcal{M}(\mathcal{N})$ and  let's consider  the following  decomposition of the solution   $u=w+z$ with $W=(w,w_t)$ satisfying the problem
	\begin{eqnarray}\left\{\begin{array}{l}\label{prob-H2-sol-1}
			w_{tt} -\Delta w + g(w_t) + f(w) = h,\quad\mbox{in}\quad\Omega \times (s,T)  \\
			w=0,\quad \mbox{on}\quad\Gamma,\\
			(w(s),w_t(s))=(u^*,0),
		\end{array}\right.\end{eqnarray}
	where {\bf $u^*\in H^2$ is a stationary solution of \eqref{P}}, and $Z=(z,z_t)$ satisfying the following problem
	\begin{eqnarray}\left\{\begin{array}{l}\label{prob-H2-sol-2}
			z_{tt} -\Delta z  +g(u_t )- g(w_t)=-[\,f(u)-f(w)\,], \quad\mbox{in}\quad\Omega \times (s,T),\\
			z=0,\quad\mbox{on}\quad\Gamma,\\
			(z(s),z_t(s)) = (u(s)-u^*,u_t(s)).
		\end{array}\right.\end{eqnarray}

	\paragraph{\bf{Estimate 1}.} The estimates performed below are \enquote{formal}. However, they can be easily justified by considering Galerkin approximations.
	
	Taking the inner product of the Eq. in \eqref{prob-H2-sol-1} with $-\Delta w_t$ in $H^0$, yields
	\begin{eqnarray}\label{reg-H2-1}
		\begin{aligned}
			&\frac{d}{dt}\left[\,\frac{1}{2}||W(t)||^2_{\mathcal{H}_1}+\frac{1}{2}\int_{\Omega}f'(w)|\nabla w|^2dx+\int_{\Omega}h\Delta wdx\,\right]+\int_{\Omega}g'(w_t)|\nabla w_t|^2dx\\
			&\quad=\frac{1}{2}\int_{\Omega}f''(w)w_t|\nabla w|^2dx.
		\end{aligned}
	\end{eqnarray}
	On the other hand, taking the inner product of the Eq. in \eqref{prob-H2-sol-1} with $-\alpha\Delta w$,  with $\alpha\in (0,1)$, we get
	\begin{eqnarray}\label{reg-H2-2}
		\begin{aligned}
			&\alpha\frac{d}{dt}\int_{\Omega}\nabla w_t\nabla wdx+\alpha\|\Delta w(t)\|^2+\alpha\int_{\Omega}f'(w)|\nabla w|^2dx+\alpha\int_{\Omega}h\Delta wdx\\
			&=\alpha\|\nabla w_t(t)\|^2-\alpha\int_{\Omega}g'(w_t)\nabla w\nabla w_tdx.
		\end{aligned}
	\end{eqnarray}
	Then, combining \eqref{reg-H2-1} and \eqref{reg-H2-2}, we obtain
	\begin{eqnarray}\label{reg-H2-12}
		\begin{aligned}
			&\frac{d}{dt}\mathcal{E}_w(t)+\alpha\mathcal{E}_w(t)+\frac{\alpha}{2}\|\Delta w(t)\|^2+\int_{\Omega}g'(w_t)|\nabla w_t|^2dx\\
			&=\frac{3\alpha}{2}\|\nabla w_t(t)\|^2+\frac{1}{2}\int_{\Omega}f''(w)w_t|\nabla w|^2dx-\alpha\int_{\Omega}g'(w_t)\nabla w\nabla w_tdx+\alpha\int_{\Omega}\nabla w_t\nabla wdx.
		\end{aligned}
	\end{eqnarray}
	where
	$$\mathcal{E}_w(t):=\frac{1}{2}||W(t)||^2_{\mathcal{H}_1}+\frac{1}{2}\int_{\Omega}f'(w)|\nabla w|^2dx+\int_{\Omega}h\Delta wdx+\alpha\int_{\Omega}\nabla w_t\nabla wdx.$$
	Note that, from dissipative condition \ref{hyp-inf-f}, Holder inequality, and embedding $H^2\hookrightarrow H^1$, we can estimate the functional $\mathcal{E}_w(t)$ from below as follows
	\begin{eqnarray}\label{reg-H2-3}
		\begin{aligned}\mathcal{E}_w(t)\ge &\;\frac{1}{4}\|\Delta w(t)\|^2-\|h\|^2-\frac{K_f+\alpha^2}{2}\|\nabla w\|^2,
	\end{aligned}\end{eqnarray}
	where $K_f$ appears as in \eqref{strong-4}.
	So, we define the perturbed functional $\widetilde{\mathcal{E}}_w(t)$ by
	\begin{align*}\widetilde{\mathcal{E}}_w(t):=\mathcal{E}_w(t)+\frac{K_f+\alpha^2}{2}\|\nabla w\|^2+\|h\|^2,\end{align*}
	it follows from \eqref{reg-H2-3} that
	\begin{eqnarray}\label{reg-H2-33}\widetilde{\mathcal{E}}_w(t)\ge \frac{1}{4}\|\Delta w(t)\|^2.
	\end{eqnarray}
	Returning to \eqref{reg-H2-12}, we obtain
	\begin{eqnarray}\label{reg-H2-123}
		\begin{aligned}
			&\frac{d}{dt}\widetilde{\mathcal{E}}_w(t)+\alpha\widetilde{\mathcal{E}}_w(t)+\frac{\alpha}{2}\|\Delta w(t)\|^2+\int_{\Omega}g'(w_t)|\nabla w_t|^2dx\\
			&=\alpha\|h\|^2+\frac{3\alpha}{2}\|\nabla w_t(t)\|^2+\alpha\int_{\Omega}\nabla w_t\nabla wdx+(K_f+\alpha^2)\int_{\Omega}\nabla w_t\nabla wdx\\
			&\quad+\frac{\alpha(K_f+\alpha^2)}{2}\|\nabla w(t)\|^2+\frac{1}{2}\int_{\Omega}f''(w)w_t|\nabla w|^2dx-\alpha\int_{\Omega}g'(w_t)\nabla w\nabla w_tdx.
		\end{aligned}
	\end{eqnarray}
	Now let's estimate the terms on the right-hand side of \eqref{reg-H2-123}. First, from embedding $H^2\hookrightarrow H^1$, \eqref{regularity-trajec-attrac-1}, and \eqref{regularity-trajec-attrac-2}, it follows:
	\begin{align*}
		&\alpha\|h\|^2+\frac{3\alpha}{2}\|\nabla w_t\|^2+\alpha\int_{\Omega}\nabla w_t\nabla wdx+(K_f+\alpha^2)\int_{\Omega}\nabla w_t\nabla wdx+\frac{\alpha(K_f+\alpha^2)}{2}\|\nabla w\|^2\\
		&\le \alpha\|h\|^2+\frac{3\alpha R^2_1}{2}+\alpha R_1R+(K_f+\alpha^2)R_1R+\frac{\alpha(K_f+\alpha)}{2}R^2=:C_{R,R_1,h}.
	\end{align*}
	From Assumption (\ref{hyp_f''}), Holder inequality with $\frac{1}{2}+\frac{1}{6}+\frac{1}{3}=1$, embeddings $H^2\hookrightarrow H^1\hookrightarrow L^6(\Omega)$, \eqref{regularity-trajec-attrac-1}, and \eqref{reg-H2-33}, we have
	\begin{eqnarray*}
		\frac{1}{2}\int_{\Omega}f''(w)w_t|\nabla w|^2dx
		&\le& \frac{\alpha}{2}\widetilde{\mathcal{E}}_w(t)+\mathbf{d}(t,w)\widetilde{\mathcal{E}}_w(t),
	\end{eqnarray*}
	where $\mathbf{d}(t,w):=C_{R}\|w_t(t)\|^6_6$.
	Finally, from Young's inequality, Assumption \eqref{hyp_g'}, Holder inequality with $\frac{2}{3}+\frac{1}{3}=1$, embedding $H^1\hookrightarrow L^6(\Omega)$, and \eqref{regularity-trajec-attrac-2}, we have
	\begin{eqnarray*}
		-\alpha\int_{\Omega}g'(w_t)\nabla w\nabla w_tdx
		&\le& \int_{\Omega}g'(w_t)|\nabla w_t|^2dx+\alpha^2C'_{R_1}\|\Delta w(t)\|^2.
	\end{eqnarray*}
	Thus, using the last three inequalities, we conclude from \eqref{reg-H2-123} that
	\begin{eqnarray}\label{reg-H2-4}
		\frac{d}{dt}\widetilde{\mathcal{E}}_w(t)+\frac{\alpha}{2}\widetilde{\mathcal{E}}_w(t)+\alpha\left(\frac{1}{2}-\alpha C'_{R_1}\right)\|\Delta w(t)\|^2\le \mathbf{d}(t,w)\widetilde{\mathcal{E}}_w(t)+C_{R,R_1,h},
	\end{eqnarray}
	Now, choosing $\alpha$ small enough such that $\frac{1}{2}-\alpha C'_{R_1}\ge0$, it follows from \eqref{reg-H2-4} that
	\begin{eqnarray}\label{reg-H2-5}
		\begin{aligned}
			\frac{d}{dt}\widetilde{\mathcal{E}}_w(t)+\frac{\alpha}{2}\widetilde{\mathcal{E}}_w(t)\le \mathbf{d}(t,w)\widetilde{\mathcal{E}}_w(t)+ C_{R,R_1,h}.
		\end{aligned}
	\end{eqnarray}
	Multiplying \eqref{reg-H2-5} by the integrating factor $e^{\frac{\alpha}{2}t}$ and integrating from $s$ to $t$, we obtain
	\begin{eqnarray}\label{reg-H2-6}
		\begin{aligned}
			\widetilde{\mathcal{E}}_w(t)e^{\frac{\alpha}{2}t}\le e^{\frac{\alpha}{2}s}\widetilde{\mathcal{E}}_w(s)+\int_s^t\mathbf{d}(\tau,w)e^{\frac{\alpha}{2}\tau}\widetilde{\mathcal{E}}_w(\tau)d\tau+ \frac{2C_{R_1,R_2,h}}{\alpha}e^{\frac{\alpha}{2}t}.
		\end{aligned}
	\end{eqnarray}
	Using that $\mathbf{d}(\cdot,w)\in L^1(s,t)$ and \eqref{reg-H2-33}, it follows from \eqref{reg-H2-6} that
	\begin{eqnarray}\label{reg-H2-7}
		\begin{aligned}
			\|\Delta w(t)\|^2e^{\frac{\alpha}{2}t}\le 4e^{\|\mathbf{d}\|_{L^1}}\left[\, e^{\frac{\alpha}{2}s}\widetilde{\mathcal{E}}_w(s)+\frac{ 2C_{R,R_1,h}}{\alpha}e^{\frac{\alpha}{2}t}\right].
		\end{aligned}
	\end{eqnarray}
	Note that, from \eqref{est-8} (Lemma \ref{lemma-stat-sol}) [regularity of steady states], we have
	\begin{align}\label{reg-H2-8}\widetilde{\mathcal{E}}_w(s)\le C\|\Delta u^*\|^2+\frac{1}{2}\|h\|^2\le C\varrho+\frac{1}{2}\|h\|^2=:\varrho_0.\end{align}
	Thus, taking the limit in \eqref{reg-H2-7} with $s\to -\infty$, we obtain
	
	\begin{eqnarray}\label{reg-H2-9}
		\begin{aligned}
			\|\Delta w(t)\|^2\le \frac{8C_{R,R_1,h} e^{\|\mathbf{d}\|_{L^1(s,t)}}}{\alpha}.
		\end{aligned}
	\end{eqnarray}
	As for \eqref{reg-H2-8}, $\widetilde{\mathcal{E}}_w(s)\le \varrho_0$ independent on $s$, we conclude from \eqref{reg-H2-9} that the full trajectory $\{W(t)=(w(t),w_t(t)):t\in \mathbb{R}\}$ has $H^2\times H^1$ regularity.

	\paragraph{\bf{Estimate 2}.}
	Our goal now is to show that the original trajectory  $U(t)$ coincides with $W(t)$ for $t\in (-\infty, T_0]$ for some $T_0<0$. In fact, multiplying \eqref{prob-H2-sol-2} by $z_t$ and integrating over $\Omega\times[s,t]$, we have
	\begin{eqnarray}\label{est-z-1}
		\begin{aligned}
			&E_z(t)+\int_s^t\int_{\Omega}\left[\,g(u_t)-g(w_t)\,\right]z_tdxd\tau=E_z(s)\\
			&\quad+\underbrace{K_f\int_s^t\int_{\Omega}z_tzdxd\tau}_{\mathrm{J}_1}+\underbrace{\frac{1}{2}\int_s^t\int_{\Omega}\int_0^1f''(\theta u+(1-\theta)w)[u_t+(1-\theta)w_t]d\theta |z|^2dxd\tau}_{\mathrm{J}_2}.
		\end{aligned}
	\end{eqnarray}
	where
	$$E_z(t)=\frac{1}{2}\|z_t\|^2+\frac{1}{2}\|\nabla z\|^2+\frac{1}{2}\int_{\Omega}\int_0^1f'(\theta u+(1-\theta)w)d\theta|z|^2dx+\frac{K_f}{2}\|z\|^2.$$
	From embedding $H^1\hookrightarrow H^0$, we have
	\begin{align*}
		\left|\mathrm{J}_1\right|\le\frac{K_f}{\lambda_1^{1/2}}\int_s^TE_z(\tau)d\tau.
	\end{align*}
	From H\"older inequality with $\frac{2}{3}+\frac{1}{6}+\frac{1}{6}=1$, embedding $H^1\hookrightarrow L^6(\Omega)$, and uniform estimates \eqref{regularity-trajec-attrac-2}, we have
	\begin{align*}
		\left|\mathrm{J}_2\right|
		\le C_{R,R_1}\int_s^tE_z(\tau)d\tau.
	\end{align*}
	Substituting $|\mathrm{J}_1|$ and $|\mathrm{J}_2|$ in \eqref{est-z-1}, we get
	\begin{align}\label{est-z-4}
		E_z(t)\le E_z(s)+C'_{R,R_1}\int_s^tE_z(\tau)d\tau.
	\end{align}
	Applying Gronwall's lemma in \eqref{est-z-4}, we obtain
	\begin{align}\label{est-z-44}
		E_z(t)\le e^{C'_{R_1,R_2}(t-s)}E_z(s).
	\end{align}
	Taking $t=s+r$, with $r>0$, from \eqref{est-z-44}, we have
	\begin{align}\label{est-z-5}
		E_z(s+r)\le e^{C'_{R,R_1}r}E_z(s).
	\end{align}
	Then, using that $\frac{1}{2}||Z(t)||^2_{\mathcal{H}}\le E_z(t)\le C||Z(t)||^2_{\mathcal{H}}$, it follows from \eqref{est-z-5} that
	\begin{align}\label{est-z-7}
		||Z(s+r)||^2_{\mathcal{H}}\le e^{C'_{R,R_1}r}||Z(s)||^2_{\mathcal{H}}.
	\end{align}
	From \eqref{est-z-7} and (\ref{prob-H2-sol-2}) , we conclude that
	\begin{align*}
		||Z(s+r)||^2_{\mathcal{H}}\longrightarrow_{s\rightarrow -\infty} 0.
	\end{align*}
	where we have used $Z(s) \rightarrow 0 , s\rightarrow -\infty $.
	That is, for any $\varepsilon>0$ there exists a $T_0=s_0+r<0$ such that
	\begin{align*}
		||Z(t)||^2_{\mathcal{H}}<\varepsilon,\quad \forall t\le T_0.
	\end{align*}
	This proves that the trajectories $U(t)$ and $W(t)$ coincide for $t\in(-\infty,T_0]$.

	\subsubsection*{Step 3: Forward propagation of the regularity}
	Let us now consider the problem \eqref{P} with initial data $U(T_0)\in H^2\times H^1$ for $T_0<0$ as obtained in Step 2. That is,
	\begin{eqnarray*}\left\{\begin{array}{l}
			u_{tt} -\Delta u + g(u_t) + f(u) = h,\quad\mbox{in}\quad\Omega \times (T_0,\infty)  \\
			w=0,\quad \mbox{on}\quad\Gamma,\\
			(u(T_0),u_t(T_0))=(u^*_0,u^*_1)\in H^2\times H^1.
		\end{array}\right.\end{eqnarray*}
	Proceeding analogously to {\it Estimate 1} of Step 2, we find the following inequality (analogous to \eqref{reg-H2-7})
	\begin{align*}
		\|\Delta u(t)\|^2\le 4e^{\|d\|_{L^1(T_0,t)}}e^{-\varpi(t-T_0)}\widetilde{\mathcal{E}}_{u}(T_0)+4e^{\|d\|_{L^1(T_0,t)}}K_{R,R_1,h},\quad \forall t\ge T_0,
	\end{align*}
	where $\varpi>0$, $d(t,u)=C\|u_t\|^6_6$, and $\widetilde{\mathcal{E}}_{u}(T_0)\le C(||U(T_0)||_{\mathcal{H}_1})<\infty$.
	This implies that $u$ is uniformly bounded in $H^2$ for all $t\in[T_0,\infty)$. Therefore, by combining steps 2 and 3 we can conclude that $u\in L^{\infty}(\mathbb{R};H^2)$ and there exists $R_2>0$ such that
	$\|u(t)\|_{H^2}\le R_2,$ $\forall t\in \mathbb{R}.$ This concludes the proof of  maximal regularity, hence of Part {\bf(II)}.
	
	\section{Appendix-Ultimate Dissipativity}\label{appendix}
	To conclude this work, we present the proof of the existence of an absorbing set for the system $(\mathcal{H}, S_t)$, which ensures that $(\mathcal{H}, S_t)$ is ultimately dissipative.
	\begin{proposition}
		\label{Prop-absorbing-set}
		Under the assumptions of Theorem \ref{theo-global}, the dynamical system $(\mathcal{H},S_t)$ has an absorbing set $\mathcal{B}$ in $\mathcal{H}$.
	\end{proposition}
	\begin{proof}
		Let $B$ be a bounded subset in $\mathcal{H}$ and let $U(t)=S_tU_0$ be a weak solution of problem \eqref{P} with $U_0\in \mathcal{H}$. In order, let us consider the modified functional
		$ \mathcal{J}(U(t)):=E(U(t))+L,$
		where $L:=\frac{1}{\omega\lambda_1}\|h\|^2+C_{\nu}|\Omega|$ and $E(U(t))$ is defined in \eqref{functional-energy}. Hence
		\begin{align}\label{abs-0}\mathcal{J}(U(t))\ge \frac{\omega}{4} ||U(t)||^2_{\mathcal{H}}.\end{align}
		Using the energy equality \eqref{energy-equality} with $s=t$ and $t=t+1$, we obtain
		\begin{align}\label{absorbing2}
			\int_{t}^{t+1}\int_{\Omega}g(u_t)u_tdxds=E(U(t))-E(U(t+1))=\mathcal{J}(U(t))-\mathcal{J}(U(t+1))=:\mathcal{D}(t).
		\end{align}
		From Assumption \eqref{hyp_g'}, and 
		and  H\"older inequality with $\frac{2}{3}+\frac{1}{3}=1$, it follows from \eqref{absorbing2} that
		\begin{align}\label{absorbing3}
			\int_t^{t+1}\|u_t(s)\|^2ds\le |\Omega|^{2/3}\left(\int_t^{t+1}\|u_t(s)\|^{6}_{6}ds\right)^{1/3}\le \frac{|\Omega|^{2/3}}{\kappa_0^{1/3}}\mathcal{D}(t)^{1/3}.
		\end{align}
		Furthermore, using Mean Value Theorem (MVT) and \eqref{absorbing3} there exist $t_1\in[t,t+\frac{1}{4}]$ and $t_2\in[t+\frac{3}{4},t+1]$ such that
		\begin{align}\label{absorbing3'}
			\|u_t(t_j)\|^2\le 4\int_t^{t+1}\|u_t(s)\|^2ds\le \frac{4|\Omega|^{2/3}}{\kappa_0^{1/3}}\mathcal{D}(t)^{1/3},\quad j=1,2.
		\end{align}
		Next, multiplying Eq. (\ref{P}) by $u(t)$ and integrating over $[t_1,t_2]\times \Omega$, we have
		\begin{eqnarray}\label{D}
			\begin{aligned}
				&\int_{t_1}^{t_2}E(s)ds+\frac{1}{2}\int_{t_1}^{t_2}\|\nabla u(s)\|^2ds
				=\;\frac{3}{2}\int_{t_1}^{t_2}\|u_t(s)\|^2ds-\int_{t_1}^{t_2}\int_{\Omega}g(u_t)udxds\\
				&\;+\int_{\Omega}u(t_1)u_t(t_1)dx-\int_{\Omega}u(t_2)u_t(t_2)dx
				\;+\int_{t_1}^{t_2}\left[\int_{\Omega}F(u)dx-\int_{\Omega}f(u)udx\right]ds.
			\end{aligned}
		\end{eqnarray}
		From Assumption \eqref{hyp_f2}, we get
		\begin{align}\label{abs-1}
			\int_{t_1}^{t_2}\left[\int_{\Omega}F(u)dx-\int_{\Omega}f(u)udx\,\right]ds\le \frac{\nu}{2}\int_{t_1}^{t_2}\int_{\Omega}|u|^2dxds\le \frac{\nu}{2\lambda_1}\int_{t_1}^{t_2}\|\nabla u(s)\|^2ds.
		\end{align}
		Then, taking the sum of $L$ on both sides of \eqref{D} and, using the definition of $\mathcal{J}(U(t))$ and \eqref{abs-1}, we get
		\begin{eqnarray}\label{DD}
			\begin{aligned}
				\int_{t_1}^{t_2}\mathcal{J}(U(s))ds+\frac{\omega}{2}\int_{t_1}^{t_2}\|\nabla u(s)\|^2ds
				\le &\;\frac{3}{2}\int_{t_1}^{t_2}\|u_t(s)\|^2ds-\underbrace{\int_{t_1}^{t_2}\int_{\Omega}g(u_t)udxds}_{X_1}\\
				&\;+\underbrace{\int_{\Omega}u(t_1)u_t(t_1)dx-\int_{\Omega}u(t_2)u_t(t_2)dx}_{X_2}+L.
			\end{aligned}
		\end{eqnarray}
		Using Assumption \eqref{hyp_g'}, H\"older and Young inequalities, embeddings $H^1\hookrightarrow L^{6}(\Omega)\hookrightarrow H^0$, \eqref{abs-0}, and \eqref{absorbing3}, we have
		\begin{eqnarray*}
			X_1
			&\le&\kappa_1\left[\left(\int_{t}^{t+1}\|u_t\|^2ds\right)^{1/2}\left(\int_t^{t+1}\|u\|^2ds\right)^{1/2}+\left(\int_{t_1}^{t_2}\|u_t\|^{6}_6ds\right)^{5/6}\left(\int_t^{t+1}\|u\|^6_{6}ds\right)^{1/6}\right]\\
			&\le&\frac{\kappa_1}{\lambda_1^{1/2}}\left[\left(\int_{t}^{t+1}\|u_t\|^2ds\right)^{1/2}\left(\int_t^{t+1}\|\nabla u\|^2ds\right)^{1/2}
			+\left(\int_{t_1}^{t_2}\|u_t\|^{6}_6ds\right)^{5/6}\left(\int_t^{t+1}\|\nabla u\|^6ds\right)^{1/6}\right]\\
			&\le&\frac{2\kappa_1}{(\omega\lambda_1)^{1/2}}\left[\left(\frac{|\Omega|^{2/3}}{\kappa_0^{1/3}}\mathcal{D}(t)^{1/3}\right)^{1/2}\sup_{t\le s\le t+1}\mathcal{J}(U(s))^{1/2}
			+\left(\frac{1}{\kappa_0}\mathcal{D}(t)\right)^{5/6}\sup_{t\le s\le t+1}\mathcal{J}(U(s))^{1/2}\right]\\
			&\le&\frac{2\kappa_1^2|\Omega|^{2/3}}{\varepsilon\omega\lambda_1\kappa_0^{1/3}}\mathcal{D}(t)^{1/3}+ \frac{2\kappa_1}{\varepsilon\omega\lambda\kappa_0^{5/3}}\mathcal{D}(t)^{5/3}+\varepsilon\sup_{t\le s\le t+1}\mathcal{J}(U(s)).
		\end{eqnarray*}
		From embedding $H^1\hookrightarrow H^0$ and \eqref{absorbing3'}, we get
		\begin{align*}
			X_2\le \|u(t_1)\|\|u_t(t_1)\|+\|u(t_2)\|\|u_t(t_2)\|
			\le  \frac{16|\Omega|^{2/3}}{\varepsilon\omega\kappa_0^{1/3}\lambda_1}\mathcal{D}(t)^{1/3}+\varepsilon \sup_{t\le s\le t+1}\mathcal{J}(U(s)).
		\end{align*}
		Returning to \eqref{D}, we obtain
		\begin{align}\label{D'}
			\int_{t_1}^{t_2}\mathcal{J}(U(s))ds
			\le C_0\mathcal{D}(t)^{1/3}+C_1\mathcal{D}(t)^{5/3}+2\varepsilon\sup_{t\le s\le t+1}\mathcal{J}(U(s))+L,
		\end{align}
		where $C_{0}:=\frac{|\omega|^{2/3}}{\kappa_0^{1/3}}+\frac{2\kappa_1^2|\Omega|^{2/3}}{\varepsilon\omega\lambda_1\kappa_0^{1/3}}+\frac{16|\Omega|^{2/3}}{\varepsilon\omega\kappa_0^{1/3}\lambda_1}$, $C_1=\frac{2\kappa_1}{\varepsilon\omega\lambda\kappa_0^{5/3}}$.
		
		Using that $\mathcal{J}$ is decreasing and MVT for integrals, there exists $\tau\in [t_1,t_2]$ such that
		\begin{align}\label{absorbing}
			\mathcal{J}(U(t+1))\le	\mathcal{J}(U(\tau))=\frac{1}{t_2-t_1}\int_{t_1}^{t_2}\mathcal{J}(U(s))ds\le 2\int_{t_1}^{t_2}\mathcal{J}(U(s))ds.\end{align}
		Hence, from \eqref{absorbing2}, \eqref{absorbing}, and \eqref{D'}, we get
		\begin{eqnarray*}
			\mathcal{J}(U(t))&=&\mathcal{J}(U(t+1))+\mathcal{D}(t)	\le 2\int_{t_1}^{t_2}\mathcal{J}(U(s))ds+\mathcal{D}(t)\\
			&\le&2C_0\mathcal{D}(t)^{1/3}+2C_1\mathcal{D}(t)^{5/3}+\mathcal{D}(t)+4\varepsilon\sup_{t\le s\le t+1}\mathcal{J}(U(s))+2L.
		\end{eqnarray*}
		Choosing $\varepsilon=\frac{1}{8}$ and using $\mathcal{J}(U(t))=\sup_{t\le s\le t+1}\mathcal{J}(U(s))$, we obtain
		\begin{align}\label{ee}
			\mathcal{J}(U(t))^{3}\le 4\left[\,4C_{0}+4C_1\mathcal{D}(t)^{4/3}+ 2\mathcal{D}(t)^{2/3}\,\right]^{3}\left[\,\mathcal{J}(U(t))-\mathcal{J}(U(t+1))\,\right]+[2^{8/3}L]^{3}.
		\end{align}
		Using the definition of $\mathcal{D}(t)$ and the energy equality \eqref{energy-equality} we obtain that $|\mathcal{D}(t)|\le 2|E(U_0)|$. Hence,  $$C_0:=\sup_{U_0\in B}4\left[\,4C_{0}+4C_1\mathcal{D}(t)^{4/3}+ 2\mathcal{D}(t)^{2/3}\,\right]^{3}<\infty.$$
		Thus, \eqref{ee} can be rewritten as
		\begin{align}\label{abs-x}
			\mathcal{J}(U(t))^{3}\le C_0\left[\,\mathcal{J}(U(t))-\mathcal{J}(U(t+1))\,\right]+\left[2^{8/3}L\right]^{3}.
		\end{align}
		Applying Nakao's lemma \cite[Lemma 2.1]{Nakao-06} to \eqref{abs-x}, we conclude that
		\begin{align}\label{abs-x2}
			\mathcal{J}(U(t))\le \left(2C_0^{-1}(t-1)^++\mathcal{J}(U(0))^{-2}\right)^{-\frac{1}{2}}+2^{8/3}L.
		\end{align}
		Using \eqref{abs-0}, it follows from \eqref{abs-x2} that
		\begin{align}\label{abs-xx}
			\lim\sup_{t\rightarrow \infty} 	||U(t)||^2_{\mathcal{H}}\le  \lim\sup_{t\rightarrow \infty} \frac{4}{\omega}\left(2C_0^{-1}(t-1)^++\mathcal{J}(U(0))^{-2}\right)^{-\frac{1}{2}}+\frac{2^{14/3}L}{\omega}  \leq \frac{2^{14/3}L}{\omega}.
		\end{align}
		Choosing $R^2>\frac{2^{14/3}L}{\omega}$, we conclude, from \eqref{abs-xx} 
		that the ball $\mathcal{B}=\left\{U\in \mathcal{H}|\,||U||_{\mathcal{H}}<R\right\}$
		is an absorbing set in $\mathcal{H}$, for the dynamical system $(\mathcal{H},S_t)$.
	\end{proof}
	\ifdefined\xxxxx
	\begin{remark}
		In the particular case, when $h\equiv 0$ and $C_{\nu}=0$ in \eqref{hyp_f2} we have that $\mathcal{J}(U(t))=E(U(t))$ and $L=0$. Thus, in \eqref{abs-xxx} we obtain that
		\begin{align*}\lim\sup_{t\rightarrow +\infty}||U(t)||^2_{\mathcal{H}}\equiv 0.
		\end{align*}
		That is, the dynamical system $(\mathcal{H},S_t)$ possesses, in its natural topological phase space $\mathcal{H}$ a compact global attractor $\mathfrak{A}=\{(0,0)\in \mathcal{H}\}$ (the set of a single point).
	\end{remark}
	\fi

	\ifdefined\xxxxxx
	\noindent{\it Proof.}{\bf (III)}
	Let $\mathfrak{B}$ be the absorbing set obtained in the Proposition \ref{Prop-absorbing-set}. From Proposition \ref{stabilizability-estimate}, the system $(\mathcal{H},S_t)$ is quasi-stable on $\mathfrak{B}$. Now, let us consider the spaces $\mathcal{H}_{-s}=H^{-s+1}\times H^{-s}$ with $0<s\le 1$.
	Next, we consider the mapping $t\in[0,T]\mapsto S_tz\in \mathcal{H}_{-1}=H^0\times H^{-1}$. Let $U(t) = S_tz=z(t)$ be a weak solution with initial data $z\in \mathfrak{B}$. It follows from Theorem \ref{theo-global} that $U_t$ has regularity $U_t=(u_t,u_{tt})\in L^{6/5}(0,T;\mathcal{H}_{-1})$. Then, we have
	\begin{eqnarray*}
		||S_{t_1}z-S_{t_2}z||_{\mathcal{H}_{-1}}&\le& \int_{t_1}^{t_2}\left\|\frac{d}{dt}z(s)\right\|_{\mathcal{H}_{-1}}ds\le \left(\int_{t_1}^{t_2}\|(u_t(s),u_{tt}(s))\|_{\mathcal{H}_{-1}}^{6/5}ds\right)^{5/6}|t_2-t_1|^{1/6}\\
		&\le& C_{\mathfrak{B}}|t_2-t_1|^{1/6},\quad 0\le t_1\le t_2\le T.
	\end{eqnarray*}
	This  shows that for each $z\in \mathfrak{B}$, the map $t\mapsto S_tz$ is H\"older continuous in the extended space $\mathcal{H}_{-1}$ with exponent $\gamma=1/6$. Now, we assume that $0<s<1$. Then, using that $\mathcal{H}_{0}\hookrightarrow \mathcal{H}_{-s}\hookrightarrow \mathcal{H}_{-1}$, applying interpolation theorem in each component of $\mathcal{H}_{-s}$ and using the H\"older continuity in $\mathcal{H}_{-1}$, we obtain
	\begin{eqnarray*}
		||S_{t_1}z-S_{t_2}z||_{\mathcal{H}_{-s}}&\le& C_s||S_{t_1}z-S_{t_2}z||^{1-s}_{\mathcal{H}}||S_{t_1}z-S_{t_2}z||^{s}_{\mathcal{H}_{-1}}\\
		&\le& C_{\mathfrak{B}}|t_2-t_1|^{s/6},\quad \mbox{with}\quad s\in (0,1].
	\end{eqnarray*}
	This shows that $t\mapsto S_tz$ is H\"older continuous in the extended spaces $\mathcal{H}_{-s}$ for $s\in (0,1]$. Therefore, the existence of a generalized exponential attractor in $\mathcal{H}_{-s}$ with $0<s\le 1$, follows from Theorem \ref{Theo-CL-5}.
	\qed

	\medskip
	In what follows, we will prove Proposition \ref{stabilizability-estimate} which guarantees that the system $(\mathcal{H},S_t)$ is quasi-stable.
	\subsection{Quasi-stability property}
	\begin{remark}
		Using that $g\in C^1(\mathbb{R})$ and $g'(0)>0$, then there exists a constant $\eta>0$ such that
		$g'(s)>0$ for all $|s|<2\eta$. And since $g'>0$ in $\mathbb{R}$, there exists a constant $m>0$ such that
		$g'(s)\ge m> 0$ for all $|s|\le \eta$. On the other hand, from Assumption \eqref{hyp_g'}, we have
		$g'(s)\ge \kappa_0|s|^{4}\ge \kappa_0\eta^4>0$ for all $|s|\ge \eta$.
		Hence, choosing $\kappa_2=\inf\{m,\kappa_0\eta^4\}$, we infer
		\begin{align}\label{hyp-g'-2}g'(s)\ge \kappa_2,\quad \mbox{for all}\quad s\in \mathbb{R}.\end{align}
		Therefore, from \eqref{hyp-g'-2} and using the same arguments used in \cite[Remark 4:1]{EECT-2017} with $\gamma=4$, we obtain
		\begin{align}
			\label{hyp-g'-22}
			\left[\,g(r)-g(s)\,\right](r-s)\ge \kappa_2|s-r|^2+\frac{\kappa_0}{10}\left(|r|^4+|s|^4\right)|r-s|^2.
		\end{align}
	\end{remark}
	\begin{proposition}\label{stabilizability-estimate} Let $(u^1(t),u^1_t(t))=S_t(u^1_0,u^1_1)$ and $ (u^2(t), u^2_t(t))= S_t(u^2_0,u^2_t)$ be two mild solutions to \eqref{P} with initial data $U^1_0=(u^1_0,u^1_1), U^2_0=(u^2_0,u^2_1)$ lying in a bounded set $B\subset \mathcal{H}$. Let the hypotheses of Theorem \ref{theo-global} be valid. In addition we assume that $g'(0)>0$.  Let $z(t)=u^1(t)-u^2(t).$ Then we have the following relation
		\begin{eqnarray}\label{inequality-main}
			||S_tU^1_0-S_tU^2_0||^2_{\mathcal{H}}\le b(t)||U^1_0-U^2_0||^2_{\mathcal{H}}+c(t)\sup_{s\in [0,t]}\|u^1(s)-u^2(s)\|^2,
		\end{eqnarray}
		where $b(t)$ and $c (t)$ are nonnegative scalar functions satisfying  $b\in L^1(\mathbb{R}^+)$ with  $\displaystyle\lim_{t\rightarrow +\infty}b(t)=0$ and $c(t)$ is
		bounded on $[0,\infty]$.
	\end{proposition}
	\begin{proof}
		Let $z=u^1-u^2.$ Then $z$ satisfies the equation
		\begin{eqnarray}\left\{\begin{array}{l}\label{diff-weak-solution}\displaystyle{
					z_{tt}-\Delta z+\left[\,g(u^1_t)-g(u^2_t)\,\right]+\left[\,f(u^1)-f(u^2)\,\right]=0,\quad \mbox{in}\quad \Omega\times [0,\infty),}\\
				\displaystyle{ (z(0),z_t(0))=(u_{0}-v_{0},u_{1}-v_{1}),\quad z=0,\quad \mbox{on}\quad \Gamma.}\end{array}\right.
		\end{eqnarray}
		Using the energy identity \eqref{energy-equality} for the difference of two weak solutions, we have
		\begin{eqnarray}
			\begin{aligned}
				\label{ASP-1}
				\frac{d}{dt}E_z(t)+\int_{\Omega}D(z_t)dx=&\;\textcolor{magenta}{K_f\int_{\Omega}zz_tdx}\\
				&+\frac{1}{2}\int_{\Omega}\int_0^1f''(\theta u^1+(1-\theta)u^2)\left[\theta u^1_t+(1-\theta)u^2_t\right]d\theta|z|^2dx.
			\end{aligned}
		\end{eqnarray}
		where
		$$D(z_t):=\left[g(u^1_t)-g(u^2_t)\right]z_t$$
		and
		$$E_z:=\frac{1}{2}||S_tU^1_0-S_tU^2_0||^2_{\mathcal{H}}+\frac{1}{2}\int_{\Omega}\int_0^1f'(\theta u^1+(1-\theta)u^2)d\theta|z|^2dx+\textcolor{magenta}{\frac{K_f}{2}\|z\|^2,}$$
		\textcolor{blue}{with the constant $K_f$  obtained from dissipativity condition \eqref{hyp-inf-f} as in \eqref{ff}.}
		
		Integrating \eqref{ASP-1} from $\tau\ge 0$ to $T$, we get
		\begin{eqnarray}
			\begin{aligned}
				\label{ASP-11}
				&E_z(T)-E_z(\tau)+\int_{\tau}^T\int_{\Omega}D(z_t(s))dxds\\
				&=\textcolor{magenta}{\underbrace{K_f\int_{\tau}^T\int_{\Omega}zz_tds}_{J_1}}+\underbrace{\frac{1}{2}\int_{\tau}^T\int_{\Omega}\int_0^1f''(\theta u^1+(1-\theta)u^2)\left[\theta u^1_t+(1-\theta)u^2_t\right]d\theta|z|^2dxds}_{J_2}.
			\end{aligned}
		\end{eqnarray}\textcolor{magenta}{
			Using H\"older inequality, \eqref{hyp-g'-22}, and Young's inequality, we have
			\begin{eqnarray}
				\label{j1}
				\begin{aligned}
					|J_1|\le &\;K_f\int_{\tau}^T\|z\|\|z_t\|ds\le \frac{K_f}{\kappa_2^{1/2}}\left(\int_{\tau}^T\|z\|^2ds\right)^{1/2}\left(\int_{\tau}^T\int_{\Omega}D(z_t(s))dxds\right)^{1/2}\\
					\le &\;\frac{K_f^2}{2\kappa_2}\int_{\tau}^T\|z\|^2ds+\frac{1}{2}\int_{\tau}^T\int_{\Omega}D(z_t(s))dxds.
		\end{aligned} \end{eqnarray}}
		From Assumption \eqref{hyp_f''}, H\"older inequality with $\frac{1}{2}+\frac{1}{6}+\frac{1}{3}=1$, embedding $H^1\hookrightarrow L^6(\Omega)$, and Young's inequality, we have
		\begin{eqnarray}\label{j2}
			\begin{aligned}
				|J_2|\le& \;C\int_{\tau}^T\left[\,1+\|u^1\|^3_6+\|u^2\|^3_{6}\,\right]\left[\,\|u^1_t\|_6+\|u^2_t\|_6\,\right]\|z\|_6^2ds\\
				\le&\; C_B\int_{\tau}^T\left[\,\|u^1_t\|_6+\|u^2_t\|_6\,\right]\|\nabla z\|^2ds\\
				\le&\; \delta\int_{\tau}^TE_z(s)ds+ C_{B,\delta}\int_{\tau}^Td(s,u^1,u^2)E_z(s)ds,
			\end{aligned}
		\end{eqnarray}
		where
		$$d(t,u^1,u^2):=\|u^1_t(t)\|^6_6+\|u^2_t(t)\|^6_6.$$
		Substituting \eqref{j1} and \eqref{j2} into \eqref{ASP-11}, we find
		\begin{eqnarray}
			\begin{aligned}
				\label{ASP-111}
				E_z(T)-E_z(\tau)+\frac{1}{2}\int_{\tau}^T\int_{\Omega}D(z_t(s))dxds
				\le &\;\textcolor{magenta}{\frac{K_f^2}{2\kappa_2}\int_{\tau}^T\|z\|^2ds}+\delta\int_{\tau}^TE_z(s)ds\\
				&+ C_{B,\delta}\int_{\tau}^Td(s,u^1,u^2)E_z(s)ds.
			\end{aligned}
		\end{eqnarray}
		Now, integrating \eqref{ASP-111} from $0$ to $T$, we obtain
		\begin{eqnarray}
			\begin{aligned}
				\label{ASP-1111}
				TE_z(T)+\frac{1}{2}\int_0^T\int_{\tau}^T\int_{\Omega}D(z_t(s))dxdsd\tau
				\le&\;\int_0^TE(\tau)d\tau +\textcolor{magenta}{\frac{K_f^2}{2\kappa_2}\int_0^T\int_{\tau}^T\|z\|^2dsd\tau}\\
				&+\delta T\int_{0}^TE_z(s)ds+ TC_{B,\delta}\int_0^Td(s,u^1,u^2)E_z(s)ds.
			\end{aligned}
		\end{eqnarray}
		Next, multiplying Eq. \eqref{diff-weak-solution} by $z$ and integrating over $\Omega\times [0,T]$, we have
		\begin{eqnarray}
			\label{ASP-2}
			\begin{aligned}
				2\int_0^TE_z(s)ds=&\;2\int_0^T\|z_t\|^2ds+\textcolor{magenta}{K_f\int_0^T\|z\|^2ds}\\
				&\;-\underbrace{\left[\int_{\Omega}z_t(T)z(T)dx-\int_{\Omega}z_t(0)z(0)dx\right]}_{J_3}-\underbrace{\int_0^T\int_{\Omega}\left[g(u^1_t)-g(u^2_t)\right]zdxds}_{J_4}.
			\end{aligned}
		\end{eqnarray}
		From \eqref{hyp-g'-22}, we get
		\begin{eqnarray}\label{j0}
			2\int_0^T\|z_t\|^2ds\le \frac{2}{\kappa_2} \int_0^T\int_{\Omega}D(z_t(s))dxds.
		\end{eqnarray}
		From embedding $H^1\hookrightarrow H^0$ and using inequality \eqref{est2-7}, we have
		\begin{align}
			\label{Asymptotic-2}
			|-J_3|\le \frac{1}{\lambda_1^{1/2}}\left[\,E_z(T)+E_z(0)\,\right].
		\end{align}
		\textcolor{magenta}{
			Still, using equality \eqref{ASP-11} and estimates \eqref{j1} and \eqref{j2} (with $\tau=0$) and choosing $\delta=\frac{\lambda_1^{1/2}}{2}$, we obtain
			\begin{eqnarray}
				\label{Asymptotic-22}
				\begin{aligned}
					E_z(0)\le&\; E_z(T)+\frac{3}{2}\int_0^T\int_{\Omega}D(z_t(s))dxds+\frac{K^2_f}{2\kappa_2}\int_0^T\|z\|^2ds\\
					&+\frac{\lambda_1^{1/2}}{2}\int_{0}^TE_z(s)ds+C_{B}\int_{0}^Td(s,u^1,u^2)E_z(s)ds.
				\end{aligned}
		\end{eqnarray}}
		Then, substituting \eqref{Asymptotic-22} in \eqref{Asymptotic-2}, we obtain
		\begin{eqnarray}
			\label{j3}
			\begin{aligned}
				|-J_3|\le&\; \frac{2}{\lambda_1^{1/2}}E_z(T)+\frac{3}{2\lambda_1^{1/2}}\int_0^T\int_{\Omega}D(z_t(s))dxds+\textcolor{magenta}{\frac{K^2_f}{2\kappa_2\lambda_1^{1/2}}\int_0^T\|z\|^2ds}\\
				&+\frac{1}{2}\int_0^TE_z(s)ds+C_{B}\int_{0}^Td(s,u^1,u^2)E_z(s)ds.
			\end{aligned}
		\end{eqnarray}
		From Assumption \eqref{hyp_g'} and \eqref{hyp-g'-22}, we have
		\begin{eqnarray}\label{j4}
			\begin{aligned}
				|-J_4|=&\;\left|-\int_0^T\int_{\Omega}\int_0^1g'(\theta u^1_t+(1-\theta)u^2_t)d\theta z_tzdxds\right|\\	
				\le&\;\kappa_1\int_0^T\int_{\Omega}\int_0^1\left[\,1+|\theta u^1_t+(1-\theta)u^2_t|^4\,\right]d\theta |z_t||z|dxds\\
				\le&\;\kappa_1\int_0^T\|z_t\|\|z\|ds+8\kappa_1\int_0^T\int_{\Omega}\left[|u^1_t|^4+|u^2_t|^4\right]|z_t||z|dxds\\
				\le&\;\int_0^T\|z_t\|^2ds+\frac{\kappa_1^2}{4}\int_0^T\|z\|^2ds+8\kappa_1\int_0^T\int_{\Omega}\left[|u^1_t|^4+|u^2_t|^4\right]|z_t||z|dxds\\
				\le &\;\frac{1}{\kappa_2} \int_0^T\int_{\Omega}D(z_t(s))dxds+\frac{\kappa_1^2}{4}\int_0^T\|z\|^2ds+\underbrace{8\kappa_1\int_0^T\int_{\Omega}\left[|u^1_t|^4+|u^2_t|^4\right]|z_t||z|dxds}_{J_4'}
			\end{aligned}
		\end{eqnarray}	
		where, by Hölder's inequality with $\frac{1}{2}+\frac{1}{2}=1$ and $\frac{2}{3}+\frac{1}{3}=1$, embedding $H^1\hookrightarrow L^6(\Omega)$, \eqref{est2-7}, and \eqref{hyp-g'-22}, we have
		\begin{eqnarray}
			\label{j4'}
			\begin{aligned}
				J'_4\le&\;8\kappa_1\int_0^T\left(\int_{\Omega}\left[|u^1_t|^4+|u^2_t|^4\right]|z_t|^2dx\right)^{1/2}\left(\int_{\Omega}\left[|u^1_t|^4+|u^2_t|^4\right]|z|^2dx\right)^{1/2}ds\\
				\le&\;4\kappa_1\int_0^T\left[\int_{\Omega}\left[|u^1_t|^4+|u^2_t|^4\right]|z_t|^2dx+\int_{\Omega}\left[|u^1_t|^4+|u^2_t|^4\right]|z|^2dx\right]\\
				\le&\;4\kappa_1\int_0^T\int_{\Omega}\left[|u^1_t|^4+|u^2_t|^4\right]|z_t|^2dxds\\
				&+4\kappa_1\int_0^T\left(\int_{\Omega}\left[\,|u^1_t|^4+|u^2_t|^4\,\right]^{3/2}dx\right)^{2/3}\left(\int_{\Omega}|z|^6dx\right)^{1/3}ds\\
				\le&\;4\kappa_1\int_0^T\int_{\Omega}\left[|u^1_t|^4+|u^2_t|^4\right]|z_t|^2dxds+4\kappa_1\int_0^T\left[\,|u^1_t\|^4_6+\|u^2_t\|^4_6\,\right]\|z\|^2_6ds\\
				\le&\;4\kappa_1\int_0^T\int_{\Omega}\left[|u^1_t|^4+|u^2_t|^4\right]|z_t|^2dxds+\frac{4\kappa_1}{\lambda}\int_0^T\left[\,|u^1_t\|^4_6+\|u^2_t\|^4_6\,\right]\|\nabla z\|^2ds\\
				\le&\;4\kappa_1\int_0^T\int_{\Omega}\left[|u^1_t|^4+|u^2_t|^4\right]|z_t|^2dxds+\frac{8\kappa_1}{\lambda}\int_0^T\left[\,|u^1_t\|^4_6+\|u^2_t\|^4_6\,\right]E_z(s)ds\\
				\le&\;4\kappa_1\int_0^T\int_{\Omega}\left[|u^1_t|^4+|u^2_t|^4\right]|z_t|^2dxds+\frac{32\kappa_1^2}{\lambda^2}\int_0^Td(s,u^1,u^2)E_z(s)ds+\frac{1}{2}\int_0^TE_z(s)ds\\
				\le&\;\frac{40\kappa_1}{\kappa_0}\int_0^T\int_{\Omega}D(z_t(s))dxds+\frac{32\kappa_1^2}{\lambda^2}\int_0^Td(s,u^1,u^2)E_z(s)ds+\frac{1}{2}\int_0^TE_z(s)ds.
			\end{aligned}
		\end{eqnarray}
		Replacing \eqref{j0}, \eqref{j3}, \eqref{j4}, and \eqref{j4'} in \eqref{ASP-2}, we obtain that
		\begin{eqnarray}\label{ASP-22}
			\begin{aligned}
				\int_0^TE_z(s)ds\le&\; \frac{2}{\lambda_1^{1/2}}E_z(T)+C_0 \int_0^T\int_{\Omega}D(z_t(s))dxds\\
				&+C_1\int_0^T\|z\|^2ds+C'_B\int_0^Td(s,u^1,u^2)E_z(s)ds,
			\end{aligned}
		\end{eqnarray}
		where $C_0=\frac{2}{\kappa_2}+\frac{3}{2\lambda_1^{1/2}}+\frac{1}{\kappa_2}+\frac{40\kappa_1}{\kappa_2}$, $C_1=K_f+\frac{K_f^2}{2\kappa\lambda_1^{1/2}}+\frac{\kappa_1^2}{4}$, and $C'_B=C_B+\frac{32\kappa_1^2}{\lambda^2}$.
		Combining \eqref{ASP-1111} and \eqref{ASP-22} we obtain that
		\begin{eqnarray}
			\begin{aligned}
				\label{ASP-33}
				\left(T-\frac{4}{\lambda_1^{1/2}}\right)E_z(t)+\left(1-\delta T\right)\int_0^TE(s)ds\le&\;2C_0 \int_0^T\int_{\Omega}D(z_t(s))dxds\\
				&+\left(TC_{B,\delta}+2C'_B\right)\int_0^Td(s,u^1,u^2)E_z(s)ds\\
				&+\left(2C_1+\frac{K_fT}{2\kappa_2}\right)\int_0^T\|z\|^2ds.
			\end{aligned}
		\end{eqnarray}
		Using without loss of generality that $\frac{T}{2}-\frac{4}{\lambda_1^{1/2}}>0$ and taking $\delta$ small enough such that $1-\delta T\ge \frac{1}{2}$, it follows from \eqref{ASP-33} that
		\begin{eqnarray}
			\begin{aligned}
				\label{ASP-333}
				TE_z(T)+\int_0^TE(s)ds\le&\;4C_0 \int_0^T\int_{\Omega}D(z_t(s))dxds+2\left(TC_{B,\delta}+2C'_B\right)\int_0^Td(s,u^1,u^2)E_z(s)ds\\
				&+2\left(2C_1+\frac{K_fT}{2\kappa_2}\right)\int_0^T\|z\|^2ds.
			\end{aligned}
		\end{eqnarray}
		From \eqref{ASP-111} with $\tau=0$, we have
		\begin{align}
			\label{ASP-11111}
			\frac{1}{2}\int_{0}^T\int_{\Omega}D(z_t(s))dxds
			\le E_z(0)-E_z(T)+\delta\int_{0}^TE_z(s)ds+ C_{B,\delta}\int_{0}^Td(s,u^1,u^2)E_z(s)ds.
		\end{align}
		Substituting \eqref{ASP-11111} into \eqref{ASP-333} and choosing $\delta=\frac{1}{2C_0}$, we obtain
		\begin{eqnarray}
			\begin{aligned}
				\label{ASP-3333}
				E_z(T)\le&\;\frac{2C_0}{T+2C_0} E_z(0)+\frac{4C_0\left(TC_{B,\delta}+2C'_B\right)}{T+2C_0}\int_0^Td(s,u^1,u^2)E_z(s)ds\\
				&+\frac{4C_0\left(2C_1+\frac{K_fT}{2\kappa_2}\right)}{T+2C_0}\sup_{s\in[0,T]}\|z(s)\|^2.
			\end{aligned}
		\end{eqnarray}
		Reiterating the estimate on the intervals $[mT,(m+1)T]$ yields
		\begin{align*}
			E_z((m+1)T)\le \gamma E_z(mT)+C_{B}b_{m},\quad m=0,1,2,\cdots,
		\end{align*}
		with $0<\gamma\equiv\frac{2C_{0}}{T+2C_{0}}<1$ and $C_B=\frac{4C_0\left(TC_{B,\delta}+2C'_B\right)}{T+2C_0}+\frac{4C_0\left(2C_1+\frac{K_fT}{2\kappa_2}\right)}{T+2C_0}$, where
		$$b_m\equiv \sup_{s\in [mT,(m+1)T]}\|z(s)\|^2+\int_{mT}^{(m+1)T}d(s;u^1,u^2)E_z(s)ds.$$
		This yields
		$$E_z(mT)\le \gamma^{m}E_z(0)+c\sum_{l=1}^m\eta^{m-l}b_{l-1}.$$
		Since $\gamma<1$, using the same argument as in (\cite{chueshov-white}, Remark 3.30) along with the definition of $b_l$ we obtain that there exists $\varrho>0$ such that
		\begin{align*}
			E_z(t)\le C_1e^{-\varrho t}E_z(0)+C_2\left[\,\sup_{0\le s\le t}\|z(s)\|^2+\int_0^te^{-\varrho(t-s)}d(s,u^1,u^2)E_z(s)ds\,\right],
		\end{align*}
		for all $t\ge 0$. Therefore, applying Gronwall's lemma we find
		\begin{align*}
			E_z(t)\le \left[\,C_1e^{-\varrho t}E_z(0)+C_2\sup_{0\le s\le t}\|z(s)\|^2\,\right]e^{C_2\int_0^td(s,u^1,u^2)ds}.
		\end{align*}
		Using that $\frac{1}{2}||U^1-U^2||^2_{\mathcal{H}}\le E_z(t)\le C_B||U^1-U^2||^2_{\mathcal{H}}$, we have
		\begin{align*}
			||U^1-U^2||^2_{\mathcal{H}}\le b(t)||U^1_0-U^2_0||^2_{\mathcal{H}}+c(t)\sup_{0\le s\le t}\|z(s)\|^2.
		\end{align*}
		where
		\begin{align}b(t):=2C_BC_1e^{-\varrho t}e^{C_2\int_0^td(s,u^1,u^2)ds}\quad \mbox{and}\quad c(t):=2C_2e^{C_2\int_0^td(s,u^1,u^2)ds}.\label{definiton-c(t)}\end{align}
		Thus, using that  $d(s,u^1,u^2)=\|u^1_t\|^6_6+\|u^2_t\|^6_6\in L^1(0,t)$,  {\bf uniformly in $t>0$} we obtain
		$$b(t)\in L^1(\mathbb{R}^+)\quad \mbox{and}\quad\lim_{t\rightarrow +\infty}b(t)=0$$
		and $c(t)$ is bounded on $\mathbb{R}^+$ due to $L^1(\mathbb{R}^+)$ integrability  of $d(s,u^1,u^2) $.
		The proof of Proposition \ref{stabilizability-estimate} is  now complete.
	\end{proof}
	
	\fi

	\section*{Acknowledgments}
	The authors would like to express their sincere gratitude to the anonymous referees for their careful reading of the manuscript and for their valuable comments and suggestions. Their insightful remarks and constructive recommendations have led to substantial improvements in both the presentation and the mathematical content of this work.

\end{document}